\documentclass[12pt]{article}

\usepackage[margin=2cm]{geometry}

\usepackage{amsmath,amsthm,amsfonts,amscd,amssymb,bbm,mathrsfs,enumerate,url}
\usepackage{authblk}

\numberwithin{equation}{section}

\usepackage{braket}

\usepackage{forest}
\usepackage{comment}
\usepackage{mathtools}
\usepackage[all]{xy}

\usepackage{graphicx,tikz}
\usepackage[colorlinks = true,
            linkcolor = black,
            urlcolor  = blue,
            citecolor = black,
            anchorcolor = black,
            pdfborder={0 0 0}]{hyperref}
\makeatletter
\newcommand*{\textlabel}[2]{%
  \edef\@currentlabel{#1}% Set target label
  \phantomsection% Correct hyper reference link
  #1\label{#2}% Print and store label
}

\usetikzlibrary{matrix,arrows}
\usetikzlibrary{decorations.pathreplacing}
\usetikzlibrary{decorations.pathmorphing}
\usetikzlibrary{patterns,fadings}

\def\semicolon{;}
\def\applytolist#1{
    \expandafter\def\csname multi#1\endcsname##1{
        \def\multiack{##1}\ifx\multiack\semicolon
            \def\next{\relax}
        \else
            \csname #1\endcsname{##1}
            \def\next{\csname multi#1\endcsname}
        \fi
        \next}
    \csname multi#1\endcsname}

\def\calc#1{\expandafter\def\csname c#1\endcsname{{\mathcal #1}}}
\applytolist{calc}QWERTYUIOPLKJHGFDSAZXCVBNM;
\def\bbc#1{\expandafter\def\csname bb#1\endcsname{{\mathbb #1}}}
\applytolist{bbc}QWERTYUIOPLKJHGFDSAZXCVBNM;
\def\bffc#1{\expandafter\def\csname bf#1\endcsname{{\mathbf #1}}}
\applytolist{bffc}QWERTYUIOPLKJHGFDSAZXCVBNMqwertyuiopasdfghjklzxcvbnm;
\def\sfc#1{\expandafter\def\csname s#1\endcsname{{\sf #1}}}
\applytolist{sfc}QWERTYUIOPLKJHGFDSAZXCVBNM;
\def\rmfc#1{\expandafter\def\csname rm#1\endcsname{{\mathrm #1}}}
\applytolist{rmfc}QWERTYUIOPLKJHGFDSAZXCVBNMqwertyuiopasdfghjklzxcvbnm;
\def\frfc#1{\expandafter\def\csname fr#1\endcsname{{\mathfrak #1}}}
\applytolist{frfc}QWERTYUIOPLKJHGFDSAZXCVBNMqwertyuiopasdfghjklzxcvbnm;
\def\scfc#1{\expandafter\def\csname sc#1\endcsname{{\mathscr #1}}}
\applytolist{scfc}QWERTYUIOPLKJHGFDSAZXCVBNMqwertyuiopasdfghjklzxcvbnm;

\DeclareMathOperator{\cAd}{Ad}
\DeclareMathOperator{\id}{id}
\DeclareMathOperator{\supp}{supp}

\DeclareMathOperator{\Exp}{Exp}
\DeclareMathOperator{\Vir}{\sf{Vir}}

\DeclareMathOperator{\PSL}{PSL}

\def\bb1{\mathbbm{1}}

\def\DiffSone{\mathrm{Diff}_+(S^1)}
\def\uDiffSone{\overline{\mathrm{Diff}_+(S^1)}} %\widetilde
\def\Diff{\mathrm{Diff}}
\def\Vect{\mathrm{Vect}(S^1)}
\def\Vir{{\sf Vir}}

\def\tpi{{\tilde\pi}}
\def\<{\langle}
\def\>{\rangle}

\newtheorem{theorem}{Theorem}[section]

\newtheorem{lemma}[theorem]{Lemma}
\theoremstyle{remark}
\newtheorem{remark}[theorem]{Remark}

\title{Representations of conformal nets associated with infinite-dimensional groups}
\author[1]{Maria Stella Adamo\thanks{{\tt maria.stella.adamo@fau.de}}}
\author[2]{Luca Giorgetti \thanks{{\tt giorgett@mat.uniroma2.it}}}
\author[2]{Yoh Tanimoto\thanks{{\tt hoyt@mat.uniroma2.it}}}

\affil[1]{Department Mathematik, Friedrich-Alexander-Universit\"at Erlangen-N\"urnberg \authorcr
Cauerstrasse 11, 91058 Erlangen, Germany}
\affil[2]{Dipartimento di Matematica, Universit\`a di Roma Tor Vergata,\authorcr
   Via della Ricerca Scientifica 1, I-00133 Roma, Italy}
\date{}

\begin{document}
\maketitle
\begin{center}
 \textit{Dedicated to Karl-Hermann Neeb in occasion of his 60th birthday}
\end{center}

\begin{abstract}
We study the relation between representations of certain infinite-dimensional Lie groups and those of the associated conformal nets. For a chiral conformal net extending the net generated by the vacuum representation of a loop group or diffeomorphism group of the circle, we show that any conformal net representation induces a positive-energy representation of the corresponding group. Consequently, we prove that any representation of such a conformal net is automatically diffeomorphism covariant. Moreover, we show that the covariance cocycles of conformal net representations satisfy naturality with respect to the action of diffeomorphisms, i.e.\ the diffeomorphisms act equivariantly on the category of conformal net representations.
\end{abstract}

\section{Introduction}\label{intro}

Representation theory of infinite-dimensional Lie groups and Lie algebras is deeply intertwined with two-dimensional conformal field theory (CFT) since the seminal work of Belavin--Polyakov--Zamolodchikov \cite{BPZ84}. There, the authors exploited the fact that the chiral components of the stress-energy tensor (underlying conformal symmetry) generate an infinite-dimensional Lie algebra, namely the Virasoro algebra, to strongly constrain and to determine the correlation functions. The study of \emph{unitary} representations of the Virasoro algebra led \cite{GKO86, FQS86} to a striking classification result for two-dimensional CFTs, solely based on the positivity of the Hilbert space inner product. Another type of symmetries, namely the symmetries generated by local currents, are also described by a different family of infinite-dimensional Lie groups, the loop groups \cite{PS86}. Since \cite{KZ84CurrentAlgebra}, the theory of unitary and \emph{positive-energy} representations of such loop groups and of the associated affine Kac--Moody algebras has been applied to provide a suitable mathematical tool to determine the number and the fusion rules of superselection sectors for a notable class of CFT models, i.e., Wess--Zumino--Witten (WZW) models.

These models are specific but important, e.g., because they are the first examples where the charge structure has been fully understood.
Moreover, the Virasoro algebra is contained in a suitable sense in any CFT.
All these examples can be studied in several model-independent frameworks, ranging from formalisms which are tailored for CFTs, e.g., chiral and full vertex operator algebras \cite{Kac98, Moriwaki23}, to formalisms for more general quantum field theories, e.g., local nets of operator algebras \cite{CKLW18}, or quantum fields in the Wightman formalism \cite{AGT23Pointed} or Osterwalder--Schrader axioms \cite{AMT24OS}. In this paper, we choose to use the operator-algebraic framework of local \emph{conformal} nets \cite{GF93, CKLW18}, due to the simple mathematical description of their superselection sectors, i.e., equivalence classes of DHR representations.

A two-dimensional conformal net is a family of von Neumann algebras $\{\tilde\cA(O)\}$ (the local observables of the theory) parametrized by regions $\{O\}$
in the two-dimensional Minkowski space, acting on a common Hilbert space $\cH$ endowed with a vacuum vector $\Omega$
(the ground state of the conformal Hamiltonian),
and satisfying physically motivated axioms, such as locality,
covariance (with respect to the spacetime symmetry group), positivity of energy \cite{Haag96}. Such axioms are adapted to two-dimensional CFT from more general quantum field theory.

A prototypical example comes from a quantum field $\phi(t,x)$ smeared with test functions $f$, namely an operator-valued distribution satisfying similar
axioms, where the local algebras are defined by $\tilde\cA(O) := \{\rme^{\rmi\phi(f)}: \supp f \subset O\}''$. More generally, each local algebra $\tilde\cA(O)$ is considered as the algebra generated by the observables (self-adjoint operators) that can be measured in the spacetime region $O$.
The Einstein causality (locality) imposes that the algebras associated with spacelike separated regions should commute.
Positivity of energy (in the vacuum representation) is required to assure the stability of the theory in the vacuum state.
While the algebraic relations define a model, the same model can be in different states and by the GNS construction
one obtains different representations of the net of algebras, see \cite{SW00, Haag96} and references therein.
% for the discussions on these assumptions
A (chiral) conformal net $\{\cA(I)\}$ on $S^1$ is a family of von Neumann algebras parametrized by intervals $\{I\}$ of $S^1$
satisfying analogous axioms. These should be viewed as the building blocks of two-dimensional ``full'' conformal field theories,
where a pair of chiral conformal nets appears as subnets (``chiral components'') of the full theory \cite{Rehren00, KL04-2}.

Examples of conformal nets on $S^1$ can be constructed
by considering {\it vacuum representations} of the above mentioned infinite-dimensional Lie groups or algebras \cite{BS90, GF93}.
Once we have a conformal net, we can consider its representations \cite{FRS89, FRS92, GF93}, which we refer to as DHR representations 
after Doplicher--Haag--Roberts, who originally considered states with charges localized in bounded regions in four-dimensional Minkowski space \cite{DHR71, DHR74}.

On the one hand, the same infinite-dimensional group or algebra admits not only the vacuum representations but more general
(positive-energy, projective, unitary) representations. On the other hand, one can consider {\it charged} (non-vacuum) DHR representations of the conformal net. The purpose of this paper is to study the relation between these two representation theories.
More specifically, we will show that, for any conformal net which is covariant with respect to the group $\DiffSone$
of orientation-preserving diffeomorphisms of $S^1$
or a conformal net which extends the conformal net generated by the vacuum representation of a loop group, every DHR representation gives rise to
a positive-energy representation of the corresponding infinite-dimensional group (loop groups or $\DiffSone$).

The problem is not new, and indeed several authors studied related problems.
Most importantly, we mention \cite{Henriques19LoopGroups}, where Henriques considered the colimits of the corresponding localized
infinite-dimensional groups and constructed a functor from the representations of a loop group conformal net
to the representations of the corresponding loop group and affine Kac--Moody algebra.
Related questions have been addressed by Carpi \cite{Carpi04} for the Virasoro nets, namely the conformal nets
generated by a vacuum representation of $\DiffSone$.
Our proof strategy is different and based on the local and global structures of the infinite-dimensional groups,
see the end of Section \ref{grouprep} for more results in this direction.

As a consequence of our main result, we show that every factorial DHR representation (in particular, every irreducible or direct sum of irreducibles) is necessarily diffeomorphism covariant. This result has also been proved in \cite{Gui21Categorical}, cf.\ \cite{DFK04},
see Section \ref{covariance} for comparison. As a further consequence, we show that the covariant cocycles of DHR representations are natural with respect to the action of diffeomorphisms. This condition has been explicitly considered in \cite{DG18} and used to prove covariance of conformal net extensions both on $S^1$ and on the two-dimensional Minkowski space, \cite{KL04-1, DG18, MTW18, AGT23Pointed}. 

The proof idea of our main result goes as follows. Both for loop groups and for $\DiffSone$, one can consider the subgroups of elements localized
in proper intervals $I$ of $S^1$. We first show that there are continuous fragmentation maps from a small neighbourhood of the unit element of the corresponding group that represent its elements as the product of elements localized in intervals covering the unit circle $S^1$ and having pairwise small overlaps.
Then the vacuum representation of the given conformal net can be restricted to the localized subgroups, and any other DHR representation can be composed with the fragmentation maps. 
This allows to define a 
local \textit{multiplier representation} with the same cocycle of the vacuum representation on a sufficiently small neighbourhood of the unit element, by gluing together the fragmented elements composed with the DHR representation and by showing that the group law is preserved locally.
As our infinite-dimensional groups have universal simply connected central extensions, we can extend the local multiplier representations to (true, everywhere defined) unitary and also positive-energy representations of the central extensions.

This paper is organized as follows.
In Section \ref{preliminary}, we recall the basic concepts and examples.
Starting with the group-related notions such as local groups and representations, 2-cocycles, central extensions, projective and multiplier representations,
we summarize their constructions and basic facts.
Then, we recall the definition of conformal net on $S^1$ and of DHR representation.
Our basic examples are the loop group nets coming from the vacuum representations of loop groups,
and arbitrary conformal nets with diffeomorphism covariance.
In Section \ref{fragmentation}, we construct the continuous fragmentation maps, first for the loop groups and then for $\DiffSone$.
In Section \ref{grouprp}, we state and prove our main technical result. For a conformal net which contains either a loop group net or the Virasoro net and a given DHR representation, we construct a new local multiplier representation of the corresponding infinite-dimensional group as described above.
In Section \ref{application}, we discuss some applications of these constructions, such as conformal covariance of DHR representations and naturality of covariance cocycles.

\bigskip

We will use the following notations throughout the paper, for the sake of uniform presentation.
 \begin{itemize}
  \item $S^1$ for the circle as manifold, $\bbT$ as group.
  \item $G$ for a compact, connected, simple and simply connected Lie group, while $LG$ its loop group.
  \item $\DiffSone$ for the group of orientation-preserving diffeomorphisms of $S^1$, while
  $\overline{\DiffSone}$ its universal covering group.
  \item $\Gamma = LG, \DiffSone$, $\tilde\Gamma$ its universal central extension, so $\tilde\Gamma = \tilde LG, \Vir$ (see Section \ref{preliminary}).
  \item $e$ or $e_\Gamma$ for the unit element of the group $\Gamma = LG, \DiffSone$, and $0$ for the Lie algebras.
  \item $\cA$ for a generic conformal net, $\rho, \sigma$ its DHR representations, $\rho_0$ the vacuum representation.
  \item $\pi^\Gamma$ generic (possibly projective or multiplier) representation of $\Gamma$ (upper index to distinguish from $\pi = 3.14...$),
  $\tpi_0$ for a generic vacuum representation of $\tilde\Gamma$, $\tpi_\rho$ the induced representation by a DHR representation $\rho$ of $\tilde\Gamma$
  (see Section \ref{grouprp}).
  \item $\cU, \cV$ for open sets in $\Gamma, \tilde\Gamma$.
  \item $\frU, \frV$ for open sets in the infinite-dimensional Lie algebras (loop algebras or the Virasoro algebra).
 \end{itemize}
 
\section{Preliminaries}\label{preliminary}
\subsection{Central extensions and multiplier representations}\label{central}

\paragraph{Local homomorphisms.}
Let $\Gamma, \Gamma'$ be topological groups.
A map $f$ from $\Gamma$ to $\Gamma'$ is called a \textbf{local homomorphism}
if there is a neighbourhood $\cU$ of the unit element $e \in \Gamma$
such that if $\gamma_1, \gamma_2, \gamma_1\gamma_2 \in \cU$, then $f(\gamma_1)f(\gamma_2) = f(\gamma_1\gamma_2)$
\cite[\S 47B]{Pontryagin66}.
In particular, if $\Gamma'$ is the group $\rmU(\cH)$ of unitary operators on a Hilbert space $\cH$,
we call it a (unitary) \textbf{local representation}.
By \cite[Theorem 63]{Pontryagin66}, if $\Gamma$ is a connected, simply connected and locally path connected group (see Lemma \ref{lm:extlocal}),
then any local representation of $\Gamma$ extends to a global representation. 

Moreover, the notion of local homomorphisms can be defined
even for local groups \cite[\S 23D]{Pontryagin66}, namely sets with a local group structure.
Here, we consider a weaker notion (cf. also \cite[\S 23F]{Pontryagin66}):
a \textbf{local group} is a topological space $\Gamma$ and an inclusion of open neighbourhoods $\cV \subset \cU$ of a distinguished element $e\in\Gamma$
such that
\begin{itemize}
 \item If $a, b \in \cV$, then the product $ab \in \cU$ is defined, and the map $a,b \mapsto ab$ is continuous.
 \item If $ab, (ab)c, bc, a(bc)$ are defined, then $(ab)c = a(bc)$.
 \item $ae$ and $ea$ are defined for all $a \in \cU$ and $ae = ea = a$.
 \item If $a \in \cV$, then it has an inverse $a^{-1} \in \cU$, i.e.\! an element such that $aa^{-1} = a^{-1}a = e$,
 and the map $a \mapsto a^{-1}$ is continuous.
\end{itemize}
One can define the notions of local homomorphisms and isomorphisms also for local groups
(for isomorphism one should require that the inverse is continuous as well), cf. \cite[\S 23H]{Pontryagin66}.
Clearly, the composition of two local homomorphisms is again a local homomorphism.
In particular, the composition of a local homomorphism and a local representation is a local representation.

\paragraph{Algebraic central extensions of groups.}
For an abelian group $H$, a {\it central extension of $\Gamma$ by $H$} is a short exact sequence of group homomorphisms of the form $1\rightarrow H\xrightarrow[] {\phi} \tilde\Gamma \xrightarrow[]{\psi} \Gamma\rightarrow 1$, where $1$ denotes the trivial group, such that for every $h\in H$ and $k\in \tilde\Gamma$ one has $\phi(h)k=k\phi(h)$. From now on, by abuse of notation, we will often indicate a central extension just by $\tilde\Gamma$.  
In particular, a central extension is called {\it trivial} if $\tilde\Gamma\cong H\times \Gamma$. 

Having two central extensions of $\Gamma$ by the group $H$, specifically $1\rightarrow H\xrightarrow[] {\phi} \tilde\Gamma \xrightarrow[]{\psi} \Gamma\rightarrow 1$ and $1\rightarrow H\xrightarrow[] {\phi'} \tilde\Gamma' \xrightarrow[]{\psi'} \Gamma\rightarrow 1$, one can investigate whether they are equivalent, namely whether there exists an isomorphism of groups $\Pi: \tilde\Gamma\to \tilde\Gamma'$ such that $\Pi\circ \phi =\phi'\circ \mathrm{id}_H$
and $\mathrm{id}_\Gamma\circ \psi=\psi'\circ\Pi$. We say that a central extension {\it splits} if it is equivalent to a trivial central extension.

Every central extension gives rise to a so-called {\it 2-cocycle}, or just cocycle, on $\Gamma$ with values in $H$, i.e.\ a map $\bfc:\Gamma\times \Gamma\to H$ such that $\bfc(e_\Gamma, e_\Gamma)=e_H$ and $\bfc(\gamma_1,\gamma_2)\bfc(\gamma_1\gamma_2,\gamma_3)= \bfc(\gamma_1,\gamma_2\gamma_3)\bfc(\gamma_2,\gamma_3)$ for every $\gamma_1,\gamma_2,\gamma_3\in \Gamma$. For a section $s$ of the central extension, i.e., a map $s : \Gamma \to \tilde\Gamma$ such that $s(e_{\Gamma})=e_{\tilde\Gamma}$ and $\psi \circ s (\gamma) = \gamma$ for $\gamma \in \Gamma$, the 2-cocycle $\bfc$ is defined as $\bfc(\gamma_1,\gamma_2) := \phi^{-1}(s(\gamma_1)s(\gamma_2)s(\gamma_1\gamma_2)^{-1})$. 

On the other hand, for a given 2-cocycle $\bfc:\Gamma\times \Gamma\to H$, one can construct a central extension of $\Gamma$ by $H$ associated with $\bfc$ given by $1\rightarrow H\xrightarrow[] {\phi} H\times_\bfc \Gamma \xrightarrow[]{\psi} \Gamma\rightarrow 1$, where $H\times_\bfc \Gamma$ is supported on the cartesian product $H\times \Gamma$ and is equipped with the multiplication defined by $(h_1,\gamma_1)\cdot_\bfc(h_2,\gamma_2) := (\bfc(\gamma_1,\gamma_2)h_1h_2,\gamma_1\gamma_2)$ for $(h_1,\gamma_1),(h_2,\gamma_2)\in H\times \Gamma$.
However, for topological groups, $\bfc$ is not necessarily continuous and does not give the right topology to the central extension. For further reading, cf. \cite[Ch.3]{Schottenloher08}.

\paragraph{Topological central extensions.}
Let $1\rightarrow H\xrightarrow[] {\phi} \tilde\Gamma \xrightarrow[]{\psi} \Gamma\rightarrow 1$ 
be a central extension $\tilde\Gamma$ of a group $\Gamma$ by an abelian group $H$. If $\Gamma$, $H$ and $\tilde\Gamma$ are topological groups, $\phi$, $\psi$ are continuous homomorphisms with $\phi^{-1}$ being continuous on its image, then we will call it a {\bf topological central extension}.
If there is a global continuous section $s : \Gamma \to \tilde\Gamma$, that is, $s$ is a section, namely $s(e_{\Gamma})=e_{\tilde\Gamma}$ and $\psi \circ s (\gamma) = \gamma$ for $\gamma \in \Gamma$, which is continuous on every point of $\Gamma$,
then $\bfc(\gamma_1,\gamma_2) := \phi^{-1}(s(\gamma_1)s(\gamma_2)s(\gamma_1\gamma_2)^{-1})$ defines a continuous 2-cocycle on $\Gamma \times \Gamma$ with values in $H$. However, in general, there may not be any global continuous section.

For a neighbourhood $\cU$ of the unit element $e \in \Gamma$,
if we can take a local continuous section $s : \cU \to \tilde\Gamma$ such that $\psi \circ s (\gamma) = \gamma$, $\gamma \in \cU$,
then we can define an $H$-valued \textbf{local continuous cocycle} on $\cU \times \cU$ by $\bfc(\gamma_1,\gamma_2) := \phi^{-1}(s(\gamma_1)s(\gamma_2)s(\gamma_1\gamma_2)^{-1})$.
The map $(h, \gamma) \mapsto \phi(h)s(\gamma) \in \tilde\Gamma$ defined for $h \in H$, $\gamma \in \cU$ is a local isomorphism,
where $H \times \cU$ is equipped with a local group structure defined by $\bfc$ as above,
by taking $\cV \subset \cU$ such that $e \in \cV$ and for any $a, b \in \cV, ab, a^{-1} \in \cU$.

\paragraph{Projective representations and topological central extensions.} It is known that, for a Hilbert space $\mathcal{H}$, the unitary group $\mathrm{U}(\mathcal{H})$ can be seen as a non-trivial central extension of the unitary projective transformation group $\mathrm{U}(\mathbb{P}(\mathcal{H}))$ by $\mathbb{T}$, the unit circle group, cf. \cite[Lemma 3.4]{Schottenloher08}. Note that $\mathrm{U}(\mathcal{H})$ and $\mathrm{U}(\mathbb{P}(\mathcal{H}))$ are topological groups, the group homomorphisms in such central extension are continuous, and $\phi:\mathbb{T}\to\mathrm{U}(\mathcal{H})$ is defined as $\phi(z)=z\,\mathrm{id}_{\mathcal{H}}$, thus $\phi^{-1}$ is continuous on its image. Therefore, such central extension is also topological. 

Let $\Gamma$ be a topological group. A continuous homomorphism $\Gamma\to\mathrm{U}(\mathbb{P}(\mathcal{H}))$ with respect to the strong operator topology is called a \textbf{projective representation} of $\Gamma$. For a given everywhere defined $\bbT$-valued cocycle $c$, a \textbf{multiplier representation} of $\Gamma$ is a continuous map $\pi:\Gamma\to \mathrm{U}(\mathcal{H})$ such that $\pi(\gamma_1)\pi(\gamma_2) = c(\gamma_1, \gamma_2)\pi(\gamma_1\gamma_2)$, for every $\gamma_1,\gamma_2 \in \Gamma$. 

Let $1\rightarrow H\xrightarrow[] {\phi} \tilde\Gamma \xrightarrow[]{\psi} \Gamma\rightarrow 1$ be a topological central extension of $\Gamma$ by $H$
as above. Further assume that $\tpi$ is a continuous unitary representation of $\tilde\Gamma$ on $\cH$
such that $\tpi(\phi(H)) \subset \bbT\bb1$, where $\bb1$ denotes the identity operator on $\cH$. Then $\Gamma\ni\gamma \mapsto \Pi(\gamma) := \tpi(\tilde\gamma)/\bbT$ is clearly a well-defined continuous projective representation of $\Gamma \cong \tilde\Gamma / \phi(H)$, where $\tilde\gamma = \gamma\,\phi(H) \in \tilde\Gamma/\phi(H)$. 

Assume that there is a global continuous section $s : \Gamma \to \tilde\Gamma$, hence, in particular, $s$ is globally continuous and it holds that $\psi \circ s(\gamma) = \gamma$ for every $\gamma\in\Gamma$.
In this case, $\gamma \mapsto \tpi(s(\gamma))$ can be seen as a multiplier representation of $\Gamma$, since it satisfies $\tpi(s(\gamma_1))\tpi(s(\gamma_2)) = c(\gamma_1, \gamma_2)\tpi(s(\gamma_1\gamma_2))$, $\gamma_1,\gamma_2 \in \Gamma$, for the $\bbT$-valued cocycle $c$ defined as follows.
Let $\pi_H := \tpi \circ \phi$, then the following diagram commutes
\begin{align*}
\xymatrix
   {1 \ar[r] & H \ar[r]^{\phi} \ar[d]^{\pi_H} & \tilde\Gamma \ar[r]^{\psi} \ar[d]^{\tpi}& \Gamma\ar[r] \ar[d]^{\Pi}& 1 \\
    1 \ar[r]& \mathbb{T} \ar[r]& \mathrm{U}(\mathcal{H}) \ar[r]& \mathrm{U}(\mathbb{P}(\mathcal{H})) \ar[r]& 1 \\
   }
\end{align*}
and let $c(\gamma_1,\gamma_2) := \pi_H(\bfc(\gamma_1, \gamma_2))$, $\gamma_1,\gamma_2 \in \Gamma$, where $\bfc$ is the $H$-valued cocycle defined above.

\paragraph{Local multiplier representations.}
Let $\Gamma$ be a topological group. If there are a neighbourhood $\cU$ of the unit element $e \in \Gamma$, a continuous map $\pi$ from $\cU$ into $\rmU(\cH)$, a continuous map $c: \cU\times \cU \to \bbT$
satisfying the cocycle identity $c(\gamma_1,\gamma_2)c(\gamma_1\gamma_2,\gamma_3) = c(\gamma_1,\gamma_2\gamma_3)c(\gamma_2,\gamma_3)$ such that $\pi(\gamma_1)\pi(\gamma_2) = c(\gamma_1,\gamma_2)\pi(\gamma_1\gamma_2)$ with $\gamma_1,\gamma_2,\gamma_3 \in \cU$, we say that $\pi$ is a \textbf{local multiplier representation} of $\Gamma$. Furthermore, if $\Pi$ is a continuous projective representation of $\Gamma$ on $\cH$ such that $\pi(\gamma)/\bbT = \Pi(\gamma)$ for $\gamma\in\cU$, then $\pi$ is a local multiplier representation associated with $\Pi$.

Let $\tilde\Gamma$ be a topological central extension of $\Gamma$ by $H$, $s$ a local continuous section on $\cU$
and $\bfc$ be the corresponding continuous local $H$-valued cocycle.
If $\tpi$ is a continuous unitary representation of $\tilde\Gamma$ such that $\tpi(\phi(H)) \subset \bbT\bb1$,
then we obtain a continuous local multiplier representation $\pi := \tpi\circ s : \cU \to \rmU(\cH)$ of $\Gamma$ with continuous local $\bbT$-valued cocycle $c$ defined by
$\pi_H := \tpi\circ \phi$ and $c(\gamma_1,\gamma_2) := \pi_H(\bfc(\gamma_1, \gamma_2))$ where $\bfc$ is defined as above.

Let $\pi$ be a continuous local multiplier representation of $\Gamma$ with continuous local $\mathbb{T}$-valued cocycle $c$. Suppose further that there is a continuous representation $\pi_H$ of $H$ with values in $\bbT$ such that $c(\gamma_1,\gamma_2) = \pi_H(\bfc(\gamma_1, \gamma_2))$. Then we can see $(\pi_H\times \pi)(h, \gamma) := \pi_H(h)\pi(\gamma)$ as a continuous local unitary representation of $\tilde\Gamma$ through the local isomorphism $(h, \gamma) \mapsto \phi(h)s(\gamma) \in \tilde\Gamma$ of $H\times_\bfc\Gamma$ on $\tilde\Gamma$ (in this case, one has $H\times_\bfc \cU\simeq\phi(H)s(\cU)$). 

Summarizing and applying \cite[Theorem 63]{Pontryagin66}, in the notation above, we have the following.

\begin{lemma}\label{lm:extlocal}
 Let $\Gamma$ be a topological group, $H$ a topological abelian group, $\tilde\Gamma$
 a topological central extension of $\Gamma$ by $H$ which is connected, simply connected and locally path connected\footnote{Such assumptions on the topological central extension $\tilde\Gamma$ of $\Gamma$ by $H$ are required to apply \cite[Theorem 63]{Pontryagin66}, although the terminology \lq\lq local connectedness" in the statement of \cite[Theorem 63]{Pontryagin66} is more commonly referred to as \lq\lq local path connectedness", see \cite[\S 46J]{Pontryagin66}. }.
 
 Assume that there is a neighbourhood $\cU \subset \Gamma$ of the unit element $e\in\Gamma$ and a local continuous section $s$ which defines a 2-cocycle $\bfc$ such
 that
 $\tilde\Gamma$ and $H\times_\bfc \Gamma$ are locally isomorphic. Let $\tpi_0$ be a continuous unitary representation of $\tilde\Gamma$ such that $\tpi_0(\phi(H)) \subset \bbT\bb1$
 and let $c$ be the $\bbT$-valued cocycle of the local multiplier representation of $\Gamma$ given by $c(\gamma_1,\gamma_2) := \pi_0(\bfc(\gamma_1, \gamma_2))$, where $\pi_0 := \tpi_0 \circ s$.

For every continuous local multiplier representation $\pi_\rho$ of $\Gamma$ with the same cocycle $c$, $\pi_H\times \pi_\rho$ extends to a (global) unitary representation of $\tilde\Gamma$.
 \end{lemma}

\begin{remark}
Modifying the definitions accordingly, Lemma \ref{lm:extlocal} holds true even if the (local) representations are not unitary, since the range would be contained into invertible operators of the Hilbert space $\cH$, which is a connected group. 
\end{remark}

\subsection{Conformal nets}\label{sec:conformalnet}

Let $\mathcal{I}$ be the set of non-empty open non-dense (proper) intervals $I$ of the unit circle $S^{1}$, and denote $I' := (S^{1}\setminus I)^{\circ} \in \cI$ the interior of the complement of $I\in\mathcal{I}$. Consider the group $\mathrm{PSL}(2,\bbR) := \mathrm{SL}(2,\bbR)/\{\pm 1\}$ which acts on $S^{1}$ by M\"obius transformations. Consider also the group  $\Diff_+(S^1)$ of orientation-preserving diffeomorphisms of $S^1$. Denote $\cB(\cH)$ the algebra of bounded operators on $\cH$ and by $\rmU(\cH)$ the unitary subgroup, as in the previous section.
For a subset $\cM \subset \cB(\cH)$, denote $\cM' = \{x\in\cB(\cH): xy = yx \text{ for all } y\in\cM\}$ the commutant of $\cM$ in $\cB(\cH)$.

A {\bf M\"obius covariant net} on $S^{1}$ is a triple $(\cA, U, \Omega)$ consisting of a family of von Neumann algebras indexed by $\cI$ and acting on a common complex separable Hilbert space $\cH_0$, $\cA=\left\{\cA(I) \subset\cB(\cH_0): I\in\mathcal{I}\right\}$, a strongly continuous unitary representation $U : \mathrm{PSL}(2,\bbR) \to \rmU(\cH_0)$, and a unit vector $\Omega \in \cH_0$, satisfying the following properties
\begin{enumerate}
\item[(i)] \textbf{Isotony}: $\cA(I_{1})\subset\cA(I_{2})$, if $I_{1}\subset I_{2}$, $I_{1},I_{2}\in \mathcal{I}$.
\item[(ii)] \textbf{Locality}: $\cA(I_{1})\subset\cA(I_{2})'$, if $I_{1}\cap I_{2}=\emptyset$, $I_{1},I_{2}\in \mathcal{I}$.
\item[(iii)] \textbf{M\"obius covariance}: $U(\gamma)\cA(I)U(\gamma)^{-1}=\cA(\gamma I)$ for every $I\in\mathcal{I}$, $\gamma\in \mathrm{PSL}(2,\bbR)$.
\item[(iv)] \textbf{Positivity of energy}: The generator of the one-parameter rotation subgroup of $\mathrm{PSL}(2,\bbR)$ has non-negative spectrum.
\item[(v)] \textbf{Vacuum vector}: $\Omega$ is the unique vector (up to a phase) such that
$U(\gamma)\Omega=\Omega$ for every $\gamma\in \mathrm{PSL}(2,\bbR)$. Moreover, $(\bigvee_{I\in\mathcal{I}}\cA(I))\Omega$ is dense in $\cH_0$, where $\bigvee_{I\in\mathcal{I}}\cA(I)$ denotes the von Neumann algebra generated in $\cB(\cH_0)$ by $\cA(I)$, $I\in\mathcal{I}$.
\end{enumerate}

The algebras $\cA(I)$ are referred to as the \textbf{local algebras} of $\cA$ and $\cH_0$ as the \textbf{vacuum Hilbert space}. By abuse of notation, we will refer to a conformal net just as the local algebras $\cA(I)$. The defining representation of each $\cA(I)$ on $\cH_0$ will be denoted by $\rho_0$ and referred to as the \textbf{vacuum representation} of $\cA$. In the sequel we will identify $\cA(I)$ and $\rho_0(\cA(I))$, namely we will think of $\rho_0$ as $\cA(I)\hookrightarrow\cB(\cH_0)$. 
With these assumptions, it follows automatically that each $\cA(I)$ is a type III factor
and \textbf{Haag duality}: $\cA(I)'=\cA(I')$, cf. \cite{GF93}.

A \textbf{conformal net} on $S^1$ is a M\"obius covariant net $(\cA,U,\Omega)$ as above that satisfies in addition
\begin{enumerate}
\item[(vi)] \textbf{Diffeomorphism covariance}: The representation $U$ of $\PSL(2,\bbR)$ extends to a strongly continuous projective representation of $\Diff_+(S^1)$, again denoted by $U$, such that
\begin{align*}
U(\gamma)\cA(I)U(\gamma)^{-1}&=\cA(\gamma I), \quad\gamma\in\Diff_+(S^1),\\
U(\gamma)xU(\gamma)^{-1}&=x,\quad x\in \cA(I), \gamma\in\Diff_+(I^\prime)
\end{align*}
for every $I\in\mathcal{I}$, where $\Diff_+(I^\prime) \subset \Diff_+(S^1)$ denotes the subgroup of diffeomorphisms $\gamma$ with $\supp \gamma \subset I'$.
\end{enumerate}

For background and references we refer to \cite[Section 3]{CKLW18}, \cite{GF93}.

\subsection{DHR representations}\label{sec:DHR}

Together with the (defining) vacuum representation $\rho_0$ on $\cH_0$, a conformal net may have non-trivial ``charged'' representations.
A \textbf{representation} (or \textbf{DHR representation}) of a conformal net $\cA$ is a family
$\rho = \{\rho_I : \cA(I) \to \cB(\cH_\rho)\}_{I\in\cI}$ of normal (i.e., $\sigma$-weakly continuous) representations
of each $\cA(I)$ on a fixed Hilbert space $\cH_\rho$ with the following \textbf{compatibility condition}:
If $I_1\subset I_2$, then $\rho_{I_2}|_{\cA(I_1)} = \rho_{I_1}$.
Two DHR representations $\rho$ and $\sigma$ are called unitarily equivalent if there is a unitary $U : \cH_\rho \to \cH_\sigma$ such that $U \rho_I(x) = \sigma_I(x) U$ for every $I \in \cI$ and $x \in \cA(I)$. A DHR sector is a unitary equivalence class of DHR representations.

As every local algebra $\cA(I)$ is a type III factor, each $\rho_I$ is a spatial isomorphism (unitary equivalence), i.e., $\rho_I = \cAd V_{\rho, I}$ on $\cA(I)$ for some unitary operator $V_{\rho,I} : \cH_0 \to \cH_\rho$, where $\cAd V_{\rho, I}(\cdot):=V_{\rho, I}\cdot V^*_{\rho, I}$ on $\cA(I)$. In particular, $\cAd V^*_{\rho, I} \circ \rho_I = \rho_{0,I}$ on $\cA(I)$. In other words, for every DHR representation $\rho$ and every fixed interval $I \in \cI$, the unitarily equivalent DHR representation $\cAd V^*_{\rho, I} \circ \rho := \{\cAd V^*_{\rho, I} \circ \rho_J : \cA(J) \to \cB(\cH_0)\}_{J\in\cI}$ acts on the vacuum Hilbert space and it coincides with $\rho_{0,I}$ on $\cA(I)$. 

A DHR representation $\sigma$ with the property that $\sigma_I=\rho_{0,I}$ on $\cA(I)$ is called \lq\lq localized" in the complementary interval $I' \in \cI$, and every DHR representation is \lq\lq localizable" in each interval. Moreover, if $\sigma$ is localized in $I'$, it gives rise to endomorphisms (also called \textbf{DHR endomorphisms}) of each $\cA(K)$ associated with every bigger interval $K\in\cI$, $K \supset I'$, i.e., $\sigma_K (\cA(K)) \subset \cA(K)$.

We refer to \cite[Section IV]{GF93}, \cite[Section 2]{GL96} for background on DHR representations of conformal nets, and to \cite{DHR71, DHR74} for the original context regarding DHR representations.

\subsection{Example: Loop group nets}\label{LG}

\subsubsection{The loop group and its basic central extension}

We take a simple, compact, connected and simply connected finite-dimensional Lie group $G$.
We consider the loop group
\begin{align*}
 LG := \{f \in C^\infty(S^1, G)\} \cong \{f \in C^\infty(\bbR, G) : f(t + 2\pi) = f(t)\},
\end{align*}
with pointwise multiplication, see \cite{PS86}, \cite[Introduction]{Neeb14Semibounded}.

The loop algebra
\begin{align*}
 L\frg := \{\xi \in C^\infty(S^1, \frg)\} \cong \{\xi \in C^\infty(\bbR, \frg) : \xi(t + 2\pi) = \xi(t)\},
\end{align*}
is also equipped with the structure of the Lie algebra with the pointwise Lie bracket, namely for $\xi,\eta\in L\frg$, $[\xi,\eta]_{L\frg}(t):=[\xi(t),\eta(t)]_{\frg}$ for every $t\in S^1$.

\paragraph{Topology and manifold structure of $LG$.}

The topology in $LG$ is given by the uniform convergence of functions together with all their derivatives, namely, a sequence $\{\gamma_n\}$ converges to $\gamma$ if $\frac{d^k\gamma_n}{d\theta^k}$ converges uniformly to $\frac{d^k\gamma}{d\theta^k}$ for every $k\in\mathbb{N}$.
The pointwise exponential map $\mathrm{exp}_{LG}: L\mathfrak{g}\to LG$ is a local homeomorphism between any open neighbourhood $\cU$ of the unit element $e_{LG}$ and
$\frU :=\mathrm{exp}_{LG}^{-1}(\cU)$, which is an open neighbourhood of $0_{L\mathfrak{g}}$, cf. \cite[Section 3.2]{PS86}. 
Moreover, $LG$ is simply connected, see \cite[after Proposition (4.4.2)]{PS86}.
The loop algebra $L\frg$ is the Lie algebra of $LG$ in the sense of infinite-dimensional manifold, cf. \cite[Section 3.2]{PS86}, \cite[Appendix A]{NW09CentralExtensions}.

For each $I \in \cI$, we define the subgroup $LG_I$ of $LG$ of loops supported in $I$
(in this paper, by ``support'' of a loop $\gamma$ we mean the closure of the set of $x\in S^1$ such that $\gamma(x) \neq e_G$).

\paragraph{The basic central extension $\tilde LG$.}
As $G$ is simple, there is a symmetric invariant bilinear form $\<\cdot, \cdot\>$ on $\frg$
unique up to a scalar factor, the Killing form. We normalize that form in such a way that
$\<h_\alpha, h_\alpha\> = 2$ for all longest roots $h_\alpha$ \cite[Section 4.4]{PS86}, \cite[Definition 3.3]{Neeb14Semibounded}.
Using the normalized Killing form, we define a 2-cocycle on $L\frg$ by
\begin{align*}
 \omega(\xi, \eta) := \frac1{2\pi} \int_0^{2\pi} \<\xi(t), \eta'(t)\> dt.
\end{align*}
Furthermore, the group $\DiffSone$ acts on $L\frg$ by composition and $\omega$ is invariant under its action, i.e., for every $f \in \DiffSone$, $\omega(\xi\circ f, \eta\circ f)=\omega(\xi,\eta)$, for all $\xi,\eta\in L\frg$.
We can construct the central extension $\tilde L\frg = L\frg \times_\omega \bbR$ of $L\frg$ by $\bbR$ by taking $\tilde L\frg=L\frg\oplus\mathbb{R}$ as vector space, and by setting
\begin{align*}
 [(\xi, \alpha), (\eta, \beta)] := ([\xi, \eta], \omega(\xi,\eta)).
\end{align*}
By \cite[Theorem 4.4.1 (iv), (i)]{PS86}, the chosen normalization for the Killing form in $G$ ensures that the Lie algebra central extension $\tilde L\frg = L\frg \times_\omega \bbR$ of $L\frg$ by $\bbR$ gives rise to a Lie group central extension $1 \to \bbT \to \tilde LG \to LG \to 1$, which is the unique simply connected central extension $\tilde LG$ of $LG$ by $\bbT$ by \cite[Proposition (4.4.6)]{PS86}, the so-called {\it basic} central extension of $LG$. Moreover, $\DiffSone$ also acts on the group central extension $\tilde LG$. 

Let us describe the central extension $\tilde LG$ in more detail, cf.\ \cite[Sections 4.4, 4.5]{PS86}.
We should see $\omega$ as a skew-symmetric bilinear form on the tangent space $L\frg$ of the unit element of $LG$.
Then it defines a left-invariant $2$-form $\omega$ on $LG$ as follows (with an abuse of notations, we denote it by $\omega$ as well).
For each other point $\gamma \in LG$, any tangent vector on $T_\gamma LG$ can be seen as $\xi_\gamma$,
where $\xi \in L\frg = T_e LG$ and $\xi_\gamma$ is the image of the differential $T_e LG \to T_\gamma LG$ (the pushforward) of the left-multiplication
(cf.\! \cite[Proposition 4.4.2]{PS86}).
We define the $2$-form $\omega$ on $LG$ as the collection of 2-forms $\omega_\gamma$ on $T_{\gamma}LG$ defined by $\omega_\gamma(\xi_\gamma, \eta_\gamma) := \omega(\xi, \eta)$ for each point $\gamma\in LG$. 
In this way, $\omega$ is a left-invariant $2$-form on $LG$.
For a piecewise smooth loop $\ell$ on the manifold $LG$, let $C(\ell) := \exp (i\int_\sigma \omega) \in \bbT$,
where $\sigma$ is a piece of surface in $LG$ that has $\ell$ as the boundary. $C$ is well-defined
because $\omega/2\pi$ is integral by \cite[Theorem (4.4.1)]{PS86}. It is furthermore independent of the parametrization,
additive with respect to the concatenation of loops, and left-invariant with respect to $LG$.
We can construct a group extension as follows:
Consider the set of triples $(\gamma, p, u)$ where $\gamma \in LG$, $p$ a piecewise-smooth path from $e$ to $\gamma$,
and $u \in \bbT$. We introduce an equivalence relation between $(\gamma, p, u)$ and $(\gamma', p', u')$
by $\gamma = \gamma', u = C(p' * p^{-1})u'$, where $*$ denotes the concatenation of paths (first $p'$, then $p^{-1}$).
On the quotient by this equivalence relation, which we denote by $\tilde LG$,
we define the group operation $(\gamma_1, p_1, u_1)\cdot (\gamma_2, p_2, u_2) = (\gamma_1 \gamma_2, p_1 * (\gamma_1 p_2), u_1 u_2)$
(to show that the relation is an equivalence relation one uses the additivity of $C$, while the independence of the parametrization, additivity and left-invariance of $C$ are used to show that the group operation is well-defined).

Recall that the exponential map is a local homemorphism between any open neighbourhood $\cU$ of the unit element $e_{LG}$ and
the corresponding open neighbourhood $\frU :=\mathrm{exp}_{LG}^{-1}(\cU)$ of $0_{L\mathfrak{g}}$. Hence for $\gamma$ in a small neighbourhood $\cU$ of the unit element in $LG$, there exists a unique path $p_\gamma$ from $e$ to $\gamma$, defined using the exponential map. Therefore, the map $(\gamma, u) \mapsto (\gamma, p_\gamma, u)$ is a bijection between $\mathcal{U}\times\mathbb{T}$ in $LG \times \bbT$ and its image in $\tilde LG$. Thus, its inverse provides a local trivialization of $\tilde LG$.
The product of two elements $(\gamma_1, u_1), (\gamma_2, u_2)$ in $\cU\times\bbT$ is defined through the product of their images $(\gamma_1, p_{\gamma_1},u_1)$, $(\gamma_2, p_{\gamma_2}, u_2)$. The pre-image of such product corresponds to another element $(\gamma_1\gamma_2, \tilde c(\gamma_1, \gamma_2)u_1 u_2)$, where $\tilde c$ is a map $\cU\times \cU\to\bbT$, specific to this local trivialization. Using the associativity of the product in $\tilde LG$, it is straightforward to show that $\tilde c$ is a local 2-cocycle.
With this local trivialization, we have $(\gamma, u) = (\gamma, 1)\cdot (e, u) = (e, u)\cdot (\gamma, 1)$.
For more detail about the fibre bundle structure of $\tilde LG$, we refer to \cite[Section 4.5]{PS86}.
As $\omega$ is invariant under the action of $\DiffSone$, by \cite[Theorem 4.4.1 (i),(ii)]{PS86}, it lifts to $\tilde LG$.

Clearly, $\tilde LG$ is locally path connected.

\subsubsection{Positive-energy representations}
As for representations of the loop group $LG$, if we do not specify, we consider \textbf{strongly continuous unitary representations} \cite{Neeb14Semibounded}:
a representation $\tpi^{\tilde LG}$ of the basic\footnote{Later, we may omit the term ``basic'' and refer to $\tilde LG$ just as central extension of $LG$.} central extension $\tilde LG$ by unitary operators on a Hilbert space $\cH$ such that for all $\xi \in \cH$, $\gamma \mapsto \tpi^{\tilde LG}(\gamma)\xi$ is continuous in the norm topology of $\cH$.
Recall that $\DiffSone$ acts on $\tilde LG$, and thus $S^1$ acts on $\tilde LG$ by rotations. If one considers the universal cover $\bbR$ of $S^1$, for a strongly continuous (defined as before) unitary representation of $\tilde LG \rtimes \bbR$ on a Hilbert space $\cH$, by strong continuity we can consider the generator of the $\bbR$-part. A \textbf{positive-energy representation} of $\tilde LG$ is a strongly continuous unitary representation
of $\tilde LG \rtimes \bbR$ on a Hilbert space $\cH$ such that the generator of the $\bbR$-part has spectrum contained in $\bbR_{\ge0}$. If one identifies the elements of $G$ with the subgroup of constant loops of $LG$, then the compact group $G \times \bbT$ embeds in the group central extension $\tilde LG$, since for constant loops of $LG$ the derivative $\xi'$ vanishes, and so does the 2-form $\omega$. Furthermore, in any \emph{irreducible} positive-energy representation of $\tilde LG$, the central $\bbT$-part of $\tilde LG$ is represented by a scalar using Schur's Lemma. Moreover, the lift of the $2\pi$-rotation from $S^1$ to $\bbR$ in $\tilde LG \rtimes \bbR$ commutes with all elements of $\tilde LG$ since $\omega$ is invariant under the action of $\mathrm{Diff}_+(S^1)$, thus it must be represented by a scalar in any irreducible representation. Hence the generator of the $\bbR$-part has discrete spectrum. On elements of $G\times\bbT\hookrightarrow \tilde LG$, when any positive-energy irreducible representation is restricted to the lowest eigenspace $\cH_h$ with respect to $\bbR$, we can define an invariant of the representation, $(\lambda, k)$, where $\lambda$ is the lowest weight of $G$ on $\cH_h$, and $k \in \bbR$. All \textit{smooth}\footnote{Let $\cH^\infty$ be the subspace of vectors such that the orbit map is smooth, hence on which the elements of $L\frg$ can be applied arbitrarily many times
\cite[Section 3]{Neeb10OnDifferentiableVectors}. A representation is said to be smooth if $\cH^\infty$ is dense in $\cH$.} irreducible positive-energy representations of $\tilde LG$ are classified by an integer $k$ called
 the \emph{level} and the \emph{lowest weight} $\lambda$ \cite[Theorem 9.3.5]{PS86}.

Actually, any positive-energy representation $\tpi^{\tilde LG}$ is smooth \cite[Theorem 2.16]{Zellner16}.
This allows one to pass to the positive-energy representation $\rmd \tpi^{\tilde LG}$ of $\tilde L\frg$,
and in particular, to the subalgebra $\tilde L^{\mathrm{pol}}\frg$ spanned by $e_n \otimes x$ and $c$ (the central element), where $e_n(\theta) = e^{in\theta}$ and $x\in\frg$ (see \cite[Example 2.4]{Neeb14Semibounded} for these notations).
Each operator in the range of the restriction of $\rmd \tpi^{\tilde LG}$ to $\tilde L^{\mathrm{pol}}\frg$ preserves, and hence can be restricted to,
the subspace $\cH^{\mathrm{fin}}$ spanned by the eigenvectors of the lift of rotations $S^1$.
On $\cH^{\mathrm{fin}}$, one can also construct a representation of the Virasoro algebra $\{L_n : n \in \bbZ\}$ using the Sugawara formula \cite[\S 2]{GW84}.

As we saw above, $\tilde LG$ has a local trivialization, thus we can choose a neighbourhood $\cU$ of the unit element $e$ of $LG$
that trivializes $\tilde LG$ as $\bbT$-bundle.
A (positive-energy) representation $\tpi^{\tilde LG}$ restricted to $\tilde LG$ and composed with the inverse of the local trivialization can be seen as a \textbf{local multiplier representation} of $LG$ (``local'' refers to the restriction to $\cU$):
there is a 2-cocycle $c(\gamma_1, \gamma_2)\in\bbT$ such that $\tpi^{\tilde LG}(\gamma_1)\tpi^{\tilde LG}(\gamma_2) = c(\gamma_1, \gamma_2)\tpi^{\tilde LG}(\gamma_1 \gamma_2)$ for $\gamma_1, \gamma_2\in \cV$ in a smaller neighbourhood $\cV \subset \cU$ such that $\cV^2 \subset \cU$.

Among the irreducible positive-energy representations, for each level $k \in \bbN$ there is a special case, $\lambda = 0$,
the \textbf{vacuum representation} denoted by $\tpi_{G,k}$. In this case, the lowest weight vector $\Omega$ (the \textbf{vacuum vector})
is annihilated by $L_1, L_0, L_{-1}$. From this, we can construct a conformal net \cite[Section III.9]{GF93} by
\begin{align*}
 \cA_{G,k}(I) := \{\tpi_{G,k}(\gamma): \gamma \in \tilde LG_I\}'',\quad I\subset S^1,
\end{align*}
where $\tilde LG_I$ is the preimage of the subgroup $LG_I$ of loops supported in $I$ defined above.

Together with the (projective) representation $U$ of $\DiffSone$ generated by $\{L_n : n \in \bbZ\}$ by the Sugawara construction and the vacuum vector $\Omega$,
$(\cA_{G,k}, U, \Omega)$ is a conformal net. It is called the \textbf{loop group net} of $G$ at level $k$.

\subsection{Example: Virasoro nets}
\subsubsection{The diffeomorphism group and the Virasoro group}\label{diff}
Let us consider the group $\DiffSone$ of the orientation-preserving diffeomorphisms of the circle $S^1$.
By identifying $S^1 \cong \bbR/2\pi\bbZ$, we can realize its universal covering group as
\begin{align*}
 \overline{\DiffSone} \cong \{\gamma \in C^\infty(\bbR, \bbR) : \gamma(t + 2\pi) = \gamma(t) + 2\pi, \text{ and for all } t, \gamma'(t) > 0 \},
\end{align*}
with the composition as group operation. The unit element in $\overline{\DiffSone}$ is the identity function $e(t) = t$ for every $t\in\bbR$.
We can obtain $\DiffSone$ by taking the quotient of $\overline{\DiffSone}$ with the subgroup generated by (the lifts of) $2\pi$-rotation.

$\DiffSone$ can be given a structure of infinite-dimensional Lie group modeled on the Lie algebra of smooth vector fields
\begin{align*}
 \Vect &:= \{f \in C^\infty(S^1, \bbR)\} \cong \{f \in C^\infty(\bbR, \bbR) : f(t + 2\pi) = f(t) + 2\pi \},\\
 [f,g] &:= f'g - fg'.
\end{align*}

\paragraph{Topology and the manifold structure}
The Lie algebra $\Vect$ is equipped with the topology of the uniform convergence of all derivatives and becomes a topological (Frech\'et) vector space.
The neighbourhood system of $0$ in $\Vect$ is generated by the following sets, where $\epsilon_0, \cdots, \epsilon_n > 0$,
\begin{align*}
 \frU_{\epsilon_0, \cdots, \epsilon_n} := \{f \in \Vect: |f^{(j)}(t)| < \epsilon_j, j = 0, \cdots, n,\, t\in[0,2\pi]\}.
\end{align*}
Correspondingly, neighbourhoods of the unit element $e = \id$ (for simplicity, we denote it as $e$ instead of $e_{\uDiffSone}$) is given by
\begin{align*}
 \cU_{\epsilon_0, \cdots, \epsilon_n} := \{\gamma \in \uDiffSone: |(\gamma - e)^{(j)}(t)| < \epsilon_j, j = 0, \cdots, n,\, t\in\bbR\},
\end{align*}
where $\epsilon_0, \cdots, \epsilon_n > 0$ and $\gamma(t), e(t) = t$ are considered as elements in $C^\infty(\bbR, \bbR)$ as above\footnote{Note that $\DiffSone$ is a special case of the diffeomorphism group of a differentiable manifold $M$ of all the smooth maps $M\to M$. It can be endowed with two different topologies, called weak and strong topology. When $M$ is compact, the two topologies are equivalent, cf. \cite{Hir12}.  

However, even if $\bbR$ is not compact, the two topologies inherited by $C^{\infty}(\bbR,\bbR)$ coincide since we are considering only a subgroup of $C^{\infty}(\bbR,\bbR)$. One can show this fact by showing that a system of neighbourhoods of the identity in $\uDiffSone$ is equivalent to a system of neighbourhoods of the identity of $\DiffSone$.}.
For $\epsilon_0, \cdots, \epsilon_n$ small, the map $\gamma \mapsto \gamma - e$
gives a coordinate around $e$. By acting on it by the left-regular action,
we obtain a chart on $\uDiffSone$ that turns it in a Frech\'et manifold.

For each interval $I \subset S^1$, we consider the subgroup $\Diff_+(I) \subset \Diff_+(S^1)$ of diffeomorphisms supported in $I$ (in this paper, by ``support'' of a diffeomorphism $\gamma$ we mean the closure of the set of $x\in S^1$ such that $\gamma(x) \neq x$)

\paragraph{The central extension}
The Lie algebra $\Vect$ has a central extension $\bbR\times_{\bfc} \Vect$ with Lie bracket $[(\alpha, f), (\beta, g)] := (\bfc(f,g), [f, g])$, 
where
\begin{align*}
\bfc(f,g) := -\frac1{2\pi i}\int_0^{2\pi} f(t)(g'(t) + g'''(t))dt 
\end{align*}
cf. \cite[Definition 2.17]{Zellner16}.
If one complexifies $\Vect$ by taking instead complex-valued smooth vector fields on $S^1$, denoted by $\Vect_{\bbC}$, then one can complexify $\bbR\times_{\bfc} \Vect$ to the central extension of $\Vect_{\bbC}$ by $\bbR$. 

Since every $f\in\Vect$  can be written as a Fourier series, without loss of generality, it is enough to consider the monomials $L_n(\theta) := e^{\rmi n\theta}$. They satisfy the Virasoro algebra commutation relations, cf. \cite[Chapter 5]{Schottenloher08}, 
\begin{align*}
 [L_m, L_n] = (m-n)L_{m+n} + \frac{m(m^2-1)}{12}\delta_{m,-n}C,
\end{align*}
where $C$ is the central element in the $\bbR$-part of the complexified $\bbR\times_{\bfc} \Vect$.

Correspondingly, the group $\uDiffSone$ has a central extension, the Virasoro group $\Vir := \bbR\times_\bfB \uDiffSone$, 
where the Bott cocycle $B$ is defined for $\DiffSone \times \DiffSone$ \cite[Section 4]{NS15}, \cite[Section 3.1]{FH05}, by
\begin{align*}
 \bfB(\gamma_1, \gamma_2) &:= -\frac1{48\pi}\int_0^{2\pi} \log((\gamma_1\circ \gamma_2)'(t)) d\log\gamma'_2(t) \\
 &= -\frac1{48\pi}\int_0^{2\pi} \log((\gamma_1\circ \gamma_2)'(z)) \frac{\gamma''_2(t)}{\gamma'_2(t)}dt,
\end{align*}
that is, the group relation is given by
\begin{align*}
 (a_1, \gamma_1)\cdot (a_2, \gamma_2) = (a_1 + a_2 + \bfB(\gamma_1, \gamma_2), \gamma_1\circ\gamma_2).
\end{align*}
Also here, $\Vir\simeq \bbR\times \uDiffSone$ as manifolds, hence $\Vir$ is connected, simply connected and locally path connected, and it is called the Virasoro group.

Note that, by simple calculations, the Bott cocycle between the lifts of rotations vanish. Therefore, when restricted to (the lifts of) rotations in $\uDiffSone$, the central extension splits and we can identify (the lifts of) rotations in the Virasoro group $\Vir$.
The group $\Vir$ admits localized subgroups $\Vir_I$, that are the inverse images of the simply connected subgroup $\Diff_+(I)$
(diffeomorphisms supported in $I \subset S^1$) of $\DiffSone$
with respect to the projection $\Vir \to \uDiffSone$.

\subsubsection{Positive-energy representations}
In the case of the Virasoro group $\Vir$, we are interested in representations (continuous in the strong operator topology)
such that the subgroup of (the lifts of) rotations is represented with positive generator, so-called \emph{positive-energy} representations.
Again any such representation $\tpi^\Vir$ is smooth in the sense of the infinite-dimensional Lie groups \cite[Theorem 2.18]{Zellner16}, \cite{NS15}.
One can thus differentiate the representation $\tpi^\Vir$ and pass to the Lie algebra representation $\rmd\tpi^\Vir$ of $\bbR\times_{\bfc} \Vect$, then restrict $\rmd\tpi^\Vir$ to $\Vect$ and complexify it to obtain a positive-energy representation of the Virasoro algebra, which is a dense subalgebra of the central extension of $\Vect_{\bbC}$ by $\bbC$.
Any such \emph{irreducible} positive-energy representation of the Virasoro algebra has been classified by $c$, the value the central element is represented by,
and $h$, the lowest eigenvector of the generator of (the lifts of) rotations, \cite{GKO86, FQS86, KR87, NS15}.

If $h = 0$, we denote by $\tpi_{\Vir,c}$ the so-called \textbf{vacuum representation} of $\Vir$. The \textbf{Virasoro net} with central charge $c$ \cite[Section 2.4]{Carpi04} is then defined by
\begin{align*}
 \cA_{\Vir,c}(I) := \{\tpi_{\Vir,c}(\gamma): \gamma \in \Vir_I\}'',\quad I\subset S^1,
\end{align*}
where $\Vir_I$ is the subgroup of $\Vir$ localized in $I$ as above.
Together with the (projective) representation $U$ of $\DiffSone$ given by $\tpi_{\Vir,c}$ and the lowest weight vector $\Omega$, again called \textbf{vacuum vector}, $(\cA_{\Vir,c}, U, \Omega)$ is a conformal net.

\subsection{General setting}\label{setting}
Let either $\Gamma = LG$ or $\Gamma = \uDiffSone$.
In both cases, $\Gamma$ is a topological group, and for each proper open interval $I\subset S^1$ there is a corresponding subgroup $\Gamma_I \subset \Gamma$ of elements (either smooth loops $S^1 \to G$ or orientation-preserving diffeomorphisms $S^1 \to S^1$) supported (``localized'') in $I$.

We shall need the following property of localized subgroups: If $I_1$ and $I_2$ are two intervals with disjoint closures both contained in the interval $I_3$, then
any element in $\Gamma_{I_3}$ whose support is contained in $I_1\cup I_2$ is a product of two elements supported respectively in $I_1$ and $I_2$.
We already know that two subgroups $\Gamma_{I_1}, \Gamma_{I_2}$ of $\Gamma$ commute element-wise if $I_1 \cap I_2 = \emptyset$. In the respective simply connected central extension $\tilde\Gamma = \tilde LG$ or $\tilde\Gamma = \Vir$ (by $\bbT$ for $LG$ and by $\bbR$ for $\uDiffSone$) the preimages $\tilde \Gamma_{I_1}, \tilde\Gamma_{I_2}$
of $\Gamma_{I_1}, \Gamma_{I_2} \subset \tilde \Gamma$ also commute element-wise. This follows because in both cases $\Gamma = LG, \uDiffSone$
the 2-cocycle of the corresponding Lie algebra is localized, that is, for two elements $\xi_1, \xi_2$ (respectively $f,g$) with disjoint supports,
$\omega(\xi_1, \xi_2) = 0$ (respectively $\bfc(f,g) = 0$).

Let $(\cA, U, \Omega)$ be a conformal net as in Section \ref{sec:conformalnet}.
We assume that there is a distinguished (true\footnote{For a unitary representation $\pi$ of a group $\Gamma$, the use of the world ``true'' means that $\pi(\gamma_1)\pi(\gamma_2)=\pi(\gamma_1\gamma_2)$ for every $\gamma_1,\gamma_2\in\Gamma$.}) strongly continuous unitary representation $\tpi_0$ of $\tilde \Gamma$ such that $\tpi_0(\tilde \Gamma_I) \subset \cA(I)$
(this $\tpi_0$ is not necessarily irreducible), and assume that $\tpi_0 (\phi(H)) \subset \bbT\bb1$, where $\phi, H$ (either $H = \bbT$ for $LG$ or $H = \bbR$ for $\overline{\DiffSone}$) are as in Section \ref{central}. Furthermore, we assume that the projective representation $U$ of $\Diff_+(S^1)$ giving the diffeomorphism covariance of $\cA$
\begin{enumerate}
 \item either coincides with $\tpi_0$ when $\Gamma = \uDiffSone$, when $U$ is extended to $\Vir$ (see Remark below),
 \item or makes $\tpi_0$ covariant when $\Gamma = LG$, namely\footnote{We think of $\DiffSone$ as the quotient of $\overline{\DiffSone}$ with the subgroup generated by (the lifts of) $2\pi$-rotation.} $U([\gamma])\tpi_0(\lambda)U([\gamma])^*=\tpi_0(\lambda\circ\gamma^{-1})$, for $\gamma \in \overline{\mathrm{Diff}_+(S^1)}$, $[\gamma]\in\DiffSone$, $\lambda\in\tilde LG$.
\end{enumerate}

\begin{remark}\label{rm:UinConformalNet}
Every positive-energy projective representation $U$ of $\DiffSone$ can be decomposed as direct sum of irreducible ones, and every such representation is unitarily equivalent to the projective representation $U_{(c,h)}$ for some fixed $c>0$ and varying $h\geq0$, obtained by integrating the corresponding Virasoro module \cite[Theorem 3.2.6]{WeinerThesis}, \cite[Proposition 2.2]{Carpi04}, cf.\ \cite[Theorem A.2]{Carpi04}. If we denote the corresponding representation Hilbert space as $\mathcal{H}_{(c,h)}$, each $U_{(c,h)}$ is indeed a global multiplier representation of $\uDiffSone$ on $\mathcal{H}_{(c,h)}$ defined through unitaries acting on $\mathcal{H}_{(c,h)}$ \cite[Proposition 5.1]{FH05}. Such representation is constructed thanks to the existence of a global section of the central extension of $\uDiffSone$, which one can think of as a certain fiber bundle over $\uDiffSone$. One has the same cocycle $e^{ic\bfB(\cdot, \cdot)}$ for all summands, where $\bfB$ is the Bott cocycle of $\DiffSone$ lifted to $\uDiffSone$, cf.\! Section \ref{diff}, \cite[Section 3.1.3]{FH05}. Thus the direct sum gives a (true) unitary representation of the central extension $\Vir = \bbR \times_{\bfB} \uDiffSone$ satisfying the condition that the $\bbR$-part is represented by a scalar. This gives the claimed extension of $U$ to $\Vir$, that we denote by $\tilde U$.
Moreover, by the assumption that $U$ is a projective representation of $\DiffSone$ (rather than $\uDiffSone$), it holds that $\tilde U(0, \tau_{\overline{2\pi}})$ is a scalar, where $\tau_{\overline{2\pi}} \in \uDiffSone$ is the lift of the $2\pi$-rotation.
Lastly, note that $\tilde U(\Vir_I) \subset \cA(I')' = \cA(I)$ by conformal covariance and Haag duality, where $\Vir_I$ is the localized subgroup of $\Vir$ in the interval $I\subset S^1$, see Section \ref{sec:conformalnet}.
\end{remark}

We take a DHR representation $\rho = \{\rho_I\}$ of the net $\cA$ on $\cH_\rho$ as in Section \ref{sec:DHR}.
As each local algebra $\cA(I)$ is a type III factor, the representation $\rho_I$ of $\cA(I)$ is a spatial isomorphism.
Therefore, $\rho_I$ defines a representation of $\tilde \Gamma_I$ by unitary operators on $\cH_\rho$
by the formula 
\begin{align}\label{eq:pilocal}
\tpi_{\rho,I}(\gamma) := \rho_I(\tpi_0(\gamma)),\quad\text{for every}\;\gamma\in\tilde\Gamma_I,
\end{align}
and $\tpi_{\rho,I}(\gamma) \in \bbT$ if and only if $\tpi_0(\gamma) \in \bbT$, since $\rho_I$ is a spatial isomorphism.
By DHR compatibility of the family $\rho$, see Section \ref{sec:DHR}, the unitary $\tpi_{\rho,I}(\gamma)$ does not depend on the chosen interval $I \subset S^1$ as long as $\supp \gamma \subset I$.

The condition $\tpi_0 (\phi(H)) \subset \bbT\bb1$, where $\phi, H$ are as in Section \ref{central}, implies that $\tpi_{\rho,I}(\phi(H)) \subset \bbT\bb1$. Indeed, by definition $\tilde\Gamma_I$ is the preimage of $\Gamma_I$ in $\tilde\Gamma$ and the image of $\phi$ coincides with the kernel of $\psi$ in the central extension $\tilde\Gamma$, thus $\tpi_{\rho,I}(\phi(H))$ is well-defined. Therefore, by taking a local trivialization of $\tilde\Gamma$ (if $\Gamma = LG$, see Section \ref{LG}, if $\Gamma = \uDiffSone$, it trivializes globally), from $\tpi_{\rho,I}$ we can define a local multiplier representation of $\Gamma_I$ that we call $\pi_{\rho, I}$. In Section \ref{grouprp}, we show that $\pi_{\rho, I}$ extends to a local multiplier representation of $\Gamma$. Note that the cocycle $c$ such that $\pi_{\rho, I}(\gamma_1)\pi_{\rho, I}(\gamma_2) = c(\gamma_1, \gamma_2)\pi_{\rho, I}(\gamma_1\gamma_2)$ does not depend on $I$, as it coincides with the cocycle of the local multiplier representation $\pi_0$ of $\Gamma$ defined by $\tpi_0$.

\paragraph{Examples.}
Any conformal net $\cA$ as in Section \ref{sec:conformalnet} satisfies the condition (1) above by considering $\Gamma = \uDiffSone$ in our construction.

If $\cA$ is an extension of the loop group net $\cA_{G,k}$, then
we can also take $\Gamma = LG$ and the vacuum representation $\tpi_0 = \tpi_{G,k}$ of $\tilde LG$ at level $k$ satisfies the condition (2). Also in this case, $U$ giving the diffeomorphism covariance of $\cA$ satisfies the above condition (1).

Further, note that in both the aforementioned cases, the DHR representation $\rho$ is arbitrary.

\section{Continuous fragmentations}\label{fragmentation}
For $\Gamma = LG, \uDiffSone$, we will construct a continuous fragmentation:
Let $\{I_j\}_{j=1,2,3}$ be a cover of the unit circle by non-empty open non-dense intervals such that every point of $S^1$ is in at most two intervals at the same time as in Fig.\ \ref{fig:intervals}.
\begin{figure}[ht]
\centering
\begin{tikzpicture}[line cap=round,line join=round,>=triangle 45,x=1.0cm,y=1.0cm]
\clip(-2.43,-0.61) rectangle (6.15,5.35);
\draw(1.4,2.42) circle (2cm);
\draw [shift={(1.4,2.42)}] plot[domain=1.57:3.78,variable=\t]({1*2.2*cos(\t r)+0*2.2*sin(\t r)},{0*2.2*cos(\t r)+1*2.2*sin(\t r)});
\draw [shift={(1.4,2.42)}] plot[domain=3.62:5.79,variable=\t]({1*2.39*cos(\t r)+0*2.39*sin(\t r)},{0*2.39*cos(\t r)+1*2.39*sin(\t r)});
\draw [shift={(1.4,2.42)}] plot[domain=-0.58:1.68,variable=\t]({1*2.65*cos(\t r)+0*2.65*sin(\t r)},{0*2.65*cos(\t r)+1*2.65*sin(\t r)});
\draw (1.07,5.16)-- (1.09,4.95);
\draw (3.53,1.05)-- (3.73,0.91);
\draw (3.41,1.37)-- (3.61,1.24);
\draw (1.41,4.73)-- (1.41,4.53);
\draw (-0.45,1.05)-- (-0.27,1.17);
\draw (-0.82,1.27)-- (-0.63,1.39);
\draw (3.69,4.73) node[anchor=north west] {$I_1$};
\draw (-1.17,4.25) node[anchor=north west] {$I_2$};
\draw (1.27,0.05) node[anchor=north west] {$I_3$};
\end{tikzpicture}
\caption{The covering $\{I_j\}_{j=1,2,3}$ of the unit circle.}
 \label{fig:intervals}
\end{figure}
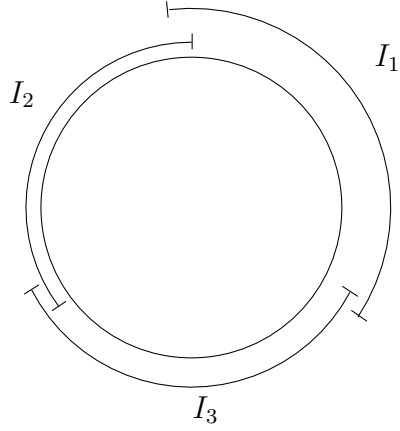
We find a neighbourhood $\cU$ of the unit element $e$ of $\Gamma$ and a triple of continuous maps $\Xi_j : \cU \to \Gamma, j=1,2,3$,
such that $\Xi_j(\gamma) \subset \Gamma_{I_j}$ and $\gamma = \Xi_1(\gamma)\Xi_2(\gamma)\Xi_3(\gamma)$.

\subsection{Continuous fragmentation of loop groups}

Recall that, for a proper open interval $I \subset S^1$, $LG_I$ is the subgroup of $LG$ of loops supported in $I$, where the support is defined by $\supp \gamma := \overline{\{x \in S^1: \gamma(x) \neq e_G\}}$ with $e_G$ the unit element of $G$.

\begin{lemma}\label{lm:fragmentation-LG}
 There is a neighbourhood $\cU$ of the unit element $e$ of $LG$ and continuous maps $\Xi_j: \cU \to LG$ such that
 \[
  \gamma = \Xi_1(\gamma)\Xi_2(\gamma)\Xi_3(\gamma)
 \]
for every $\gamma\in\cU$, $\Xi_j(\cU) \subset LG_{I_j}$
and $\Xi_j(e) = e$ for every $j = 1,2,3$. 
Furthermore, if $\gamma \in \cU$ has $\supp \gamma \subset I_1$, then $\supp \Xi_2(\gamma) \subset I_1 \cap I_2$ and $\supp \Xi_3(\gamma) \subset I_1 \cap I_3$ (but $\Xi_2(\gamma)$ and $\Xi_3(\gamma)$ need not be $e$).
\end{lemma}

\begin{proof}
For the given covering $\{I_j\}_{j=1,2,3}$ of the circle, let $\chi_1$ be a smooth function supported in $I_1$
and such that $\chi_1(x) = 1$ for $x$ in a slightly larger interval than $S^1 \setminus I_2 \cup I_3=(I_1\setminus I_2)\cap (I_1\setminus I_3)$.
Let $\chi_2$ be a smooth function supported in $(I_1 \setminus I_3) \cup I_2$ such that $\chi_2(x) = 1$ for $x$ in a slightly larger interval than $I_2 \setminus I_3$.

Let $\mathcal{V}$ be a neighbourhood of $e$ in $LG$ homeomorphic via the exponential map $\Exp$ to the open neighbourhood $V:=\mathrm{Exp}(\mathcal{V})^{-1}$ of $0$
in $L\mathfrak{g}$ (see Section \ref{LG}).
Recall that the multiplication $m_{\chi_j}$ by a smooth scalar function $\chi_j$ is continuous as a map $L\frg \to L\frg$.
We define the following decreasing open neighbourhoods of $0$ in $L\frg$:
\begin{itemize}
 \item $V_1:=\{\xi\in V:\chi_1\xi\in V\}=V\cap m^{-1}_{\chi_1}(V)$. This is open and non-empty, as it contains $0$.
 \item $V_2:=\{\xi\in V_1:\mathrm{Exp}(\chi_1\xi)^{-1}\cdot\mathrm{Exp}(\xi)\in \mathcal{V}\} = V_1 \cap \{\xi\in V:\mathrm{Exp}(\chi_1\xi)^{-1}\cdot\mathrm{Exp}(\xi)\in \mathcal{V}\}$. This is the intersection of the open set $V_1$ with the pre-image of a continuous map obtained
 as composition of (continuous) group operations in $LG$,
 a local homeomorphism $\Exp$ from $L\frg$ to $LG$ and the scalar multiplication in $L\mathfrak{g}$.
 Again this is open and non-empty because $0 \in V_2$.
 \item $V_3:=\{\xi\in V_2: \chi_2\mathrm{Exp}^{-1}\left[\mathrm{Exp}(\chi_1\xi)^{-1}\cdot \mathrm{Exp}(\xi)\right]\in V\}$, which is open and non-empty
 by a similar argument.
 \item $V_4:=\{\xi\in V_3: \mathrm{Exp}\left(\chi_2\mathrm{Exp}^{-1}\left[\mathrm{Exp}(\chi_1\xi)^{-1}\cdot \mathrm{Exp}(\xi)\right]\right)^{-1}\cdot\mathrm{Exp}(\chi_1\xi)^{-1}\cdot \mathrm{Exp}(\xi)\in \mathcal{V}\}$, which is open and non-empty
 by a similar argument.
\end{itemize}

Let $U:= V_4 \subset V$ and set $\mathcal{U}:=\mathrm{Exp}(U) \subset \cV$. Then $\cU$ is an open neighbourhood of the unit element $e \in LG$.
If $\gamma\in\mathcal{U}$, then there exists $\eta\in U$ such that $\gamma=\mathrm{Exp}(\eta)$. Define
\begin{itemize}
 \item $\Xi_1(\gamma):= \mathrm{Exp}(\chi_1\eta)$
 \item $\Xi_2(\gamma):=\mathrm{Exp}(\chi_2\mathrm{Exp}^{-1}\left[\mathrm{Exp}(\chi_1\eta)^{-1}\cdot\mathrm{Exp}(\eta)\right])$
 \item $\Xi_3(\gamma):=\Xi_2(\gamma)^{-1}\cdot\Xi_1(\gamma)^{-1}\cdot \gamma  =\mathrm{Exp}(\chi_2\mathrm{Exp}^{-1}\left[\mathrm{Exp}(\chi_1\eta)^{-1}\cdot\mathrm{Exp}(\eta)\right])^{-1}\cdot\mathrm{Exp}(\chi_1\eta)^{-1}\cdot\mathrm{Exp}(\eta)$.
\end{itemize}
With these definitions,
\begin{itemize}
 \item For $x\in S^1$ such that $\chi_1(x)=1$, one has $\Xi_1(\gamma)=\gamma$. Moreover, $\supp \Xi_1(\gamma) \subset I_1$ and $\Xi_1(\gamma)\in\mathcal{V}$ since $\eta\in V_1$
 and $\chi_1 \eta$ is supported in $I_1$.
 \item As $\eta\in U\subset V_2$, $\mathrm{Exp}(\chi_1\eta)^{-1}\cdot\mathrm{Exp}(\eta)\in\mathcal{V}$ and it is supported in $I_2\cup I_3$
 because $\mathrm{Exp}(\chi_1\eta)$ and $\mathrm{Exp}(\eta)$ coincide on $\{x \in S^1: \chi_1(x) = 1\}$, which is slightly larger than $S^1 \setminus (I_2 \cup I_3)$.
 As $\eta\in U\subset  V_3$, $\Xi_2(\gamma)\in\mathcal{V}$ and its support is contained in $I_2$.
 \item For $\eta\in U = V_4$, $\Xi_3(\gamma)\in\mathcal{V}$ and its support is contained in $I_3$ as $\Xi_2(\gamma)$ coincides with $\Xi_1(\gamma)^{-1}\cdot\gamma$ on a slightly larger set than $S^1 \setminus I_3=(I_1\setminus I_3)\cup(I_2\setminus I_3)$. Indeed, this is immediate for $x\in I_2\setminus I_3$, since $\chi_2(x)=1$ for $x$ in a slightly larger set than $I_2\setminus I_3$. Consider now $x\in I_1\setminus I_3$.
 If $x\in (I_1 \setminus I_3)\cap I_2 \subset I_2 \setminus I_3$, then $\chi_2(x)=1$.
 If $x\in (I_1 \setminus I_3)\setminus I_2 = I_1 \setminus (I_2 \cup I_3)$, then $\Xi_1(\gamma)(x)=\gamma(x)$ for every $x\in I_1\setminus(I_2\cup I_3)$,
 and therefore $(\Xi_1(\gamma)^{-1}\gamma)(x) = e_G = \Exp(0_\frg) = \Exp(\chi_2(x)0_\frg) = \Xi_2(\gamma)(x)$.
\end{itemize}
Then, it is straightforward to show that $\gamma=\Xi_1(\gamma)\Xi_2(\gamma)\Xi_3(\gamma)$.
If $\supp \gamma \subset I_1$, then by construction $\supp(\Xi_1(\gamma)^{-1}\gamma) \subset I_1$ and then
$\Xi_2(\gamma) \subset I_2 \cap I_1$. Similarly, $\Xi_3(\gamma) \subset I_3 \cap I_1$.
\end{proof}

\begin{remark}
In the proof of Lemma \ref{lm:fragmentation-LG}, $\eta$ and $\chi_j\eta$ for $j=1,2,3$ commute, namely $[\eta,\chi_j\eta]=[\chi_j\eta,\chi_k\eta]=0$
for $j,k=1,2,3$ because the Lie bracket is pointwise and $\chi_j$ are scalars.
Therefore, one can write
\begin{itemize}
 \item $\Xi_1(\gamma)=\mathrm{Exp}(\chi_1\eta)$
 \item $\Xi_2(\gamma)=\mathrm{Exp}(\chi_2(1-\chi_1)\eta)$
 \item $\Xi_3(\gamma)=\mathrm{Exp}\left[(1-(\chi_1+\chi_2(1-\chi_1))\eta\right]=\mathrm{Exp}\left[(1-\chi_1)(1-\chi_2)\eta\right]$
\end{itemize}
and observe that $\Xi_3(\gamma)$ is supported in $I_3$ since $\chi_1\neq1$ and $\chi_2\neq1$ for a slightly smaller set than $I_2\cup I_3$ and $(I_2\setminus I_3)'=I_3\cup I_2'$ respectively. These conditions hold at the same time on a smaller set in $(I_2\cup I_3)\cap(I_3\cup I_2')=I_3$. 
\end{remark}

\begin{remark}
 $\Exp$ is a local homeomorphism between $L\mathfrak g$ and $LG$, but not
 between $\Vect$ and $\DiffSone$. Hence, the observations in the previous remark are no longer true for the diffeomorphism group.
 Nevertheless, we carry out a similar construction in Section \ref{sec:frag-diff}, but without using the exponential map $\Exp$. 
\end{remark}

\subsection{Continuous fragmentation of diffeomorphisms}\label{sec:frag-diff}
Recall that, for a proper open interval $I \subset S^1$, $\Diff_+(I)$ is the subgroup of diffeomorphisms $\gamma \in \DiffSone$
such that $\supp \gamma \subset I$, where $\supp \gamma := \overline{\{x \in S^1: \gamma(x) \neq x\}}$.
For our purpose, we need to adapt the results of \cite{CDIT21PositiveEnergy} (done for Sobolev difformorphisms of $S^1$)
to $\DiffSone$, cf.\ also \cite{DFK04}, \cite{Glo07}.
For the convenience of the reader, we detail the arguments.

Let $\{I_j\}_{j =1,2,3}$ be a cover of the unit circle as in Fig.\ \ref{fig:covering}, and write $I_j = (a_j, b_j)$ for $a_j, b_j\in S^1$.
We also take slightly smaller intervals $\hat I_j = (\hat a_j, \hat b_j) \subset I_j$
that still provide a cover of $S^1$. 
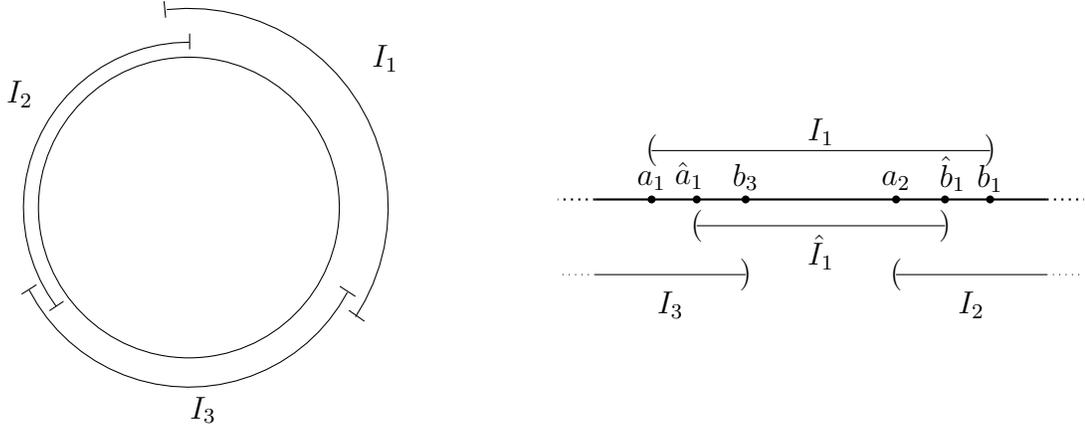
\begin{figure}[ht]
\centering
\begin{tikzpicture}[line cap=round,line join=round,>=triangle 45,x=1.0cm,y=1.0cm]
\clip(-2.43,-0.61) rectangle (6.15,5.35);
\draw(1.4,2.42) circle (2cm);
\draw [shift={(1.4,2.42)}] plot[domain=1.57:3.78,variable=\t]({1*2.2*cos(\t r)+0*2.2*sin(\t r)},{0*2.2*cos(\t r)+1*2.2*sin(\t r)});
\draw [shift={(1.4,2.42)}] plot[domain=3.62:5.79,variable=\t]({1*2.39*cos(\t r)+0*2.39*sin(\t r)},{0*2.39*cos(\t r)+1*2.39*sin(\t r)});
\draw [shift={(1.4,2.42)}] plot[domain=-0.58:1.68,variable=\t]({1*2.65*cos(\t r)+0*2.65*sin(\t r)},{0*2.65*cos(\t r)+1*2.65*sin(\t r)});
\draw (1.07,5.16)-- (1.09,4.95);
\draw (3.53,1.05)-- (3.73,0.91);
\draw (3.41,1.37)-- (3.61,1.24);
\draw (1.41,4.73)-- (1.41,4.53);
\draw (-0.45,1.05)-- (-0.27,1.17);
\draw (-0.82,1.27)-- (-0.63,1.39);
\draw (3.69,4.73) node[anchor=north west] {$I_1$};
\draw (-1.17,4.25) node[anchor=north west] {$I_2$};
\draw (1.27,0.05) node[anchor=north west] {$I_3$};
\end{tikzpicture}
\begin{tikzpicture}[path fading=north,scale=0.5]
         \draw [thick] (-6,0) --(6,0);
         \draw [thick,dotted] (-6,0) --(-7,0);
         \draw [thick,dotted] (6,0) --(7,0);
         \node at(0,1.8) {$I_1$};
         \draw [] (-3.3,-0.7) node{$($}--(3.3,-0.7)node{$)$};
         \node at(0,-1.4) {$\hat I_1$};
         \draw [] (-4.5,1.3) node{$($}--(4.5,1.3)node{$)$};
         \node at(-4.5,0.5) {$a_1$};
         \fill (-4.5,0) circle[radius=3pt];
         \node at(4.5,0.55) {$b_1$};
         \fill (4.5,0) circle[radius=3pt];
         \node at(-3.5,0.63) {$\hat a_1$};
         \fill (-3.3,0) circle[radius=3pt];
         \node at(3.5,0.68) {$\hat b_1$};
         \fill (3.3,0) circle[radius=3pt];
         \draw [dotted] (-6,-2) --(-7,-2);
         \node at(-4,-2.8) {$I_3$};
         \node at(-2,0.6) {$b_3$};
         \fill (-2,0) circle[radius=3pt];
         \draw [] (6,-2) --(2,-2)node{$($};
         \draw [] (-6,-2) --(-2,-2)node{$)$};
         \draw [dotted] (6,-2) --(7,-2);
         \node at(4,-2.8) {$I_2$};
         \node at(2,0.5) {$a_2$};
         \fill (2,0) circle[radius=3pt];
         \node at(0,-6) {};
\end{tikzpicture}
\caption{The covering of the unit circle and the subinterval of $I_1$.}
 \label{fig:covering}
\end{figure}

We now construct a fragmentation of the diffeomorphisms relative to the cover $\{I_j\}_{j =1,2,3}$, that is,
continuous maps $\{\Xi_j\}_{j=1,2,3}$ defined on a small neighbourhood $\cU$ of the unit element $e_{\DiffSone} \in \DiffSone$, such that $\Xi_j(\gamma) \in \Diff_+(I_j)$.
By continuity, the image of each $\Xi_j$ can be taken inside any other neighbourhood $\tilde \cU$ of $e$ by suitably restricting $\cU$. 
The precise statement is the following.

\begin{lemma}\label{lm:fragmentation-diff}
There is a neighbourhood $\cU$ of the unit element $e$ of $\DiffSone$ and continuous localizing maps
$\Xi_j: \cU \to$ $\Diff_+(I_j)$ such that 
\[
 \gamma = \Xi_1(\gamma)\Xi_2(\gamma)\Xi_3(\gamma)
\]
for every $\gamma\in\cU$, and such that $\Xi_j(e) = e$ and $\supp \Xi_j(\gamma) \subset I_j$ for every $j = 1,2,3$.

If $\gamma \in \cU$ has $\supp \gamma \subset I_1$, then
$\supp \Xi_2(\gamma) \subset I_1 \cap I_2$ and $\supp \Xi_3(\gamma) \subset I_1 \cap I_3$
(but $\Xi_2(\gamma)$ and $\Xi_3(\gamma)$ need not be $e$).
If $\supp \gamma \cap (a_2, b_1) = \emptyset$, then $\Xi_1(\gamma)(\theta) = \theta$ for $\theta \in (a_2, b_1)$.
\end{lemma}

\begin{proof}
The statement of continuous fragmentation is concerned with a neighbourhood of the unit element $e$,
hence we may work on the universal covering group $\uDiffSone$ and its subgroups $\overline{\Diff_+(I_j)}=\Diff_+(I_j)$ since $\Diff_+(I_j)$ is simply connected.
Accordingly, we consider the realization of $\uDiffSone$ as a subgroup of $C^\infty(\bbR, \bbR)$ (see Section \ref{diff}):
By identifying $S^1$ with $\bbR/2\pi\bbZ$, any element $\gamma \in \overline{\DiffSone}$ can be identified with a smooth function $\gamma : \bbR \to \bbR$
satisfying $\gamma'(t) > 0$ and $\gamma(t + 2\pi) = \gamma(t) + 2\pi$,
where
intervals in $S^1$ are identified with their inverse image in $\bbR$, e.g., under the covering map $t\mapsto e^{it}, \bbR \to S^1$ sending $[-\pi,\pi)$ onto $S^1$.
The unit element $e$ (for simplicity, instead of $e_{\uDiffSone}$) $\uDiffSone$ is the identity function $e(t) = t$ for every $t\in\bbR$.

Moreover, we may assume without loss of generality that, under the covering map $\bbR \to S^1$, the extreme points of $I_j$ for $j=1,2$ in $S^1$ are such that $0 < a_1 < b_1 < 2\pi$ and $0 < a_2 < b_2 < 2\pi$ modulo $2\pi$. By abuse of notation, $a_j$, $b_j$, $\hat a_j$ $\hat b_j$  will also denote points in $\bbR$. Hence, in our convention, i.e., we identified $0\in\mathbb{R}$ with a point in $I_3\setminus (I_1\cup I_2)$ and our orientation on $S^1$ is anti-clockwise, we have $0 < a_1 < \hat a_1 < \hat b_3 < b_3 < a_2 < \hat a_2 < \hat b_1 < b_1 < a_3 < \hat a_3 < \hat b_2 < b_2< 2\pi$, see Fig.\ \ref{fig:covering}.

Let us take a smooth $2\pi$-periodic function $D_{\mathrm{c},1}$ (where ``$\mathrm{c}$'' stands for center) with $D_{\mathrm{c},1}(t)=1$ for $t\in [\hat a_1, \hat b_1]$, with support in $(a_1,b_1)$, and $0 \le D_{\mathrm{c},1}(t) \le 1$.
Let $0 \le D_{\mathrm{l},1}(t) \le 1$ be another smooth $2\pi$-periodic function (where ``$\mathrm{l}$'' stands for left) with support in
$(a_1,\hat a_1)$
and $\int_0^{2\pi}D_{\mathrm{l},1}(t)dt = \int_{a_1}^{\hat a_1}D_{\mathrm{l},1}(t)dt= \frac12(\hat a_1-a_1)$,
which is possible because the interval $(a_1,\hat a_1)$ is longer than $\frac12(\hat a_1 - a_1)$.
Similarly, let $0 \le D_{\mathrm{r},1}(t) \le 1$ be a smooth $2\pi$-periodic function (where ``$\mathrm{r}$'' stands for right)
with support in $(\hat b_1,b_1)$ and $\int_0^{2\pi}D_{\mathrm{r},1}(t)dt=\int_{\hat b_1}^{b_1}D_{\mathrm{r},1}(t)dt=\frac12(b_1 - \hat b_1)$.

For $\epsilon > 0$, we consider the following neighbourhood of $e(t) = t$ in $\overline\DiffSone$
\[
 \cU_{\epsilon} := \left\{\gamma \in \overline{\DiffSone} : |\gamma(t)-e(t)|<\epsilon, |\gamma^{\prime}(t)-1|<\epsilon
 \;\text{ for }t\in[0,2\pi]\right\}.
\]
Note that $\cU_\epsilon$ is open in the uniform convergence of all the derivatives topology.

Let $\gamma \in \cU_{\epsilon}$
and define the constant $\alpha_1(\gamma)$ by
\begin{align}\label{eq:alpha}
 \alpha_1(\gamma) := \frac2{\hat a_1 - a_1}\left(\gamma(\hat a_1)-\hat a_1 - \int_0^{\hat a_1} (\gamma^{\prime}(t)-1)D_{\mathrm{c},1}(t)dt\right).
\end{align}
It follows that 
\begin{equation}\label{estalpha}
|\alpha_1(\gamma)|\leq \frac {2}{\hat a_1 - a_1} \epsilon (1+\hat{a}_1)
\end{equation} by definition of $\cU_{\epsilon}$, and
\[
 \gamma(\hat a_1)=\int_0^{\hat a_1} ((\gamma^{\prime}(t)-1)D_{\mathrm{c},1}(t)+1+\alpha_1(\gamma)D_{\mathrm{l},1}(t))dt. 
\]
Similarly, define the constant $\beta_1(\gamma)$ by
\begin{align}\label{eq:beta}
 \beta_1(\gamma) &:= \frac2{b_1 - \hat b_1}\left(\hat b_1 - \gamma(\hat b_1) - \int_{\hat b_1}^{b_1} (\gamma^{\prime}(t)-1)D_{\mathrm{c},1} (t) dt\right) \\
 &\,= \frac{-2\;\;\;}{b_1 - \hat b_1}\left(\int_0^{2\pi} ((\gamma^{\prime}(t)-1)D_{\mathrm{c},1} (t)+\alpha_1(\gamma)D_{\mathrm{l},1}(t))dt\right), \nonumber
\end{align}
where the equality follows by using the definition of $\alpha_1(\gamma)$ and the fact that $D_{\mathrm{c},1}(t) = 1$ on $[\hat a_1, \hat b_1]$.
It follows that 
\begin{align}\label{estbeta}
|\beta_1(\gamma)|\leq \frac {2}{b_1 - \hat b_1} \epsilon (1 + b_1 - \hat{b}_1)
\end{align}
 and
\begin{align*}
 b_1 = \int_0^{b_1} ((\gamma^{\prime}(t)-1)D_{\mathrm{c},1} (t)+1+\alpha_1(\gamma)D_{\mathrm{l},1}(t)+\beta_1(\gamma)D_{\mathrm{r},1}(t))dt,
\end{align*}
where all integrals cancel but $\int_0^{b_1} 1 dt = b_1$. Now, the function
\begin{align}\label{eq:gamma1}
 \gamma_1(\theta):=\int_0^\theta((\gamma^\prime (t)-1)D_{\mathrm{c},1}(t)+1+\alpha_1(\gamma)D_{\mathrm{l},1}(t)+\beta_1(\gamma)D_{\mathrm{r},1}(t))dt\,,
\end{align}
for $\theta\in\bbR$, satisfies $\gamma_1(\theta + 2\pi) = \gamma_1(\theta) + 2\pi$. Indeed, one can rewrite $\gamma_1(\theta+2\pi)$ as $\gamma_1(\theta)+\gamma_1(2\pi)$, which gives the desired results observing that $\gamma_1(2\pi)=2\pi$ and all the functions appearing, i.e., $\gamma'$, $D_{\mathrm{c},1}$, $D_{\mathrm{l},1}$, and $D_{\mathrm{r},1}$, are $2\pi$-periodic. 

The first derivative of $\gamma_1$
\begin{align*}
 \gamma'_1(\theta)= (\gamma^\prime (\theta)-1)D_{\mathrm{c},1}(\theta)+1+\alpha_1(\gamma)D_{\mathrm{l},1}(\theta)+\beta_1(\gamma)D_{\mathrm{r},1}(\theta)
\end{align*}
is positive if $\epsilon<\epsilon_1$ for a fixed $\epsilon_1$ estimated using \eqref{estalpha} and \eqref{estbeta}.
Therefore, for $\gamma\in\mathcal{U}_{\epsilon}$ for $\epsilon<\epsilon_1$, $\gamma_1$ can be regarded as an element of $\overline{\DiffSone}$.
It is also easy to check that $\gamma_1$ has the desired properties,
namely $\gamma_1(\theta)=\gamma(\theta)$ for $\theta\in [\hat{a}_1,\hat{b}_1]$ and $\gamma_1(\theta)=\theta$ for $\theta\in [0,a_1)\cup(b_1,2\pi]$,
hence $\supp \gamma_1 \subset I_1$, i.e., $\gamma_1\in \Diff_+(I_1)$.
Note that, for $\epsilon<\epsilon_1$, the assignment $\cU_{\epsilon}\rightarrow \Diff_+(I_1)$, $\gamma\mapsto\gamma_1$ is continuous
by \eqref{eq:gamma1}, \eqref{eq:alpha}, \eqref{eq:beta} in the uniform convergence of all the derivatives topology. Therefore, from now on we will assume $\epsilon<\epsilon_1$.

For $\gamma\in\cU_{\epsilon}$ we set $\Xi_1(\gamma) := \gamma_1$.
Next, we construct $\Xi_2(\gamma)$. By the continuity of $\Xi_1$ and the inverse in $ \overline{\DiffSone}$, the set $\mathcal{U}_{1,\epsilon}:=\{\gamma\in \mathcal{U}_{\epsilon}: \gamma_1^{-1}\gamma\in \mathcal{U}_{\epsilon}\}$ is a non-empty open neighbourhood of the identity in  $\overline{\DiffSone}$. Thus, for $\gamma\in\mathcal{U}_{1,\epsilon}$, we define $\gamma_2$ similarly to \eqref{eq:gamma1}, with $\gamma_1^{-1}\gamma$ instead of $\gamma$ and with $I_2$ instead of $I_1$, as follows
\begin{align}\label{eq:gamma2}
 \gamma_2(\theta):=\int_0^\theta\left[\left((\gamma_1^{-1}\gamma)^\prime (t)-1\right)D_{\mathrm{c},2}(t)+1+\alpha_2(\gamma_1^{-1}\gamma)D_{\mathrm{l},2}(t)+\beta_2(\gamma_1^{-1}\gamma)D_{\mathrm{r},2}(t)\right]dt\,,
\end{align}
for $\theta\in\bbR$, where $D_{\mathrm{c},2}$, $D_{\mathrm{l},2}$, $D_{\mathrm{r},2}, \alpha_2, \beta_2$ are defined as before but on $I_2$ instead of $I_1$
and $(\gamma_1^{-1}\gamma)(\theta)$ instead of $\gamma$. In particular, we have
\begin{align}\label{eq:alpha2}
 \alpha_2(\gamma_1^{-1}\gamma) := \frac2{\hat a_2 - a_2}
 \left((\gamma_1^{-1}\gamma)(\hat a_2)-\hat a_2 - \int_0^{\hat a_2} ((\gamma_1^{-1}\gamma)^{\prime}(t)-1)D_{\mathrm{c},2}(t)dt\right) =0
\end{align}
because, recalling that $\gamma_1(\theta)=\gamma(\theta)$ on $[\hat a_1,\hat b_1]\supset [a_2,\hat a_2]$, we have $(\gamma_1^{-1}\gamma)(\theta) = \theta$ on the support $[a_2, \hat a_2]$ of $D_{\mathrm{c},2}$
and $(\gamma_1^{-1}\gamma)(\hat a_2) = \hat a_2$.

Thus, $\gamma_2$ belongs to $\overline{\DiffSone}$, for a sufficiently small $\epsilon$. As before one has $\gamma_2(\theta) = (\gamma_1^{-1}\gamma)(\theta)$ for $\theta \in [\hat{a}_2,\hat{b}_2]$, and $\gamma_2(\theta)=\theta$ for $\theta\in[0,a_2)\cup(b_2,2\pi]$, thus $\supp \gamma_2 \subset I_2$.
In this way, we define the continuous map $\Xi_2(\gamma) := \gamma_2$.

We claim that $(\gamma_2^{-1}\gamma_1^{-1}\gamma)(\theta) = \theta$ for $\theta \in [\hat{a}_1,\hat{b}_1] \cup [\hat{a}_2,\hat{b}_2]$. Since we already know that $\gamma_2(\theta)=\gamma_1^{-1}\gamma(\theta)$ for $\theta\in[\hat{a}_2,\hat{b}_2]$,
we only need to check the claimed equality for $\theta \in [\hat a_1, \hat b_1]$.
There it holds that $(\gamma_1^{-1}\gamma)(\theta) = \theta$, therefore, we need to prove that $\gamma_2(\theta) = \theta$ on $[\hat a_1, \hat b_1]$.
To do so, observe that $[\hat a_1, \hat b_1]$ is the union of $[\hat a_1,\hat a_2]$ and $[\hat a_2, \hat b_1]$ (see Figure \ref{fig:overlap}).
One can prove the desired equality on each of these intervals using the definition \eqref{eq:gamma2} of $\gamma_2$, and using that $\gamma_1^{-1}\gamma(\theta)=\theta$ (thus $(\gamma_1^{-1}\gamma)'(\theta) = 1$) for $\theta\in[\hat{a}_1,\hat{b}_1]$.
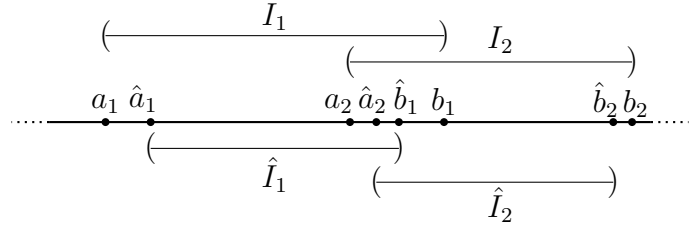
\begin{figure}[ht]
\centering
\begin{tikzpicture}[path fading=north,scale=0.5]
         \draw [thick] (-6,0) --(10,0);
         \draw [thick,dotted] (-6,0) --(-7,0);
         \draw [thick,dotted] (10,0) --(11,0);
         \node at(0,2.8) {$I_1$};
         \draw [] (-3.3,-0.7) node{$($}--(3.3,-0.7)node{$)$};
         \node at(0,-1.4) {$\hat I_1$};
         \draw [] (-4.5,2.3) node{$($}--(4.5,2.3)node{$)$};
         \node at(-4.5,0.5) {$a_1$};
         \fill (-4.5,0) circle[radius=3pt];
         \node at(4.5,0.55) {$b_1$};
         \fill (4.5,0) circle[radius=3pt];
         \node at(-3.5,0.63) {$\hat a_1$};
         \fill (-3.3,0) circle[radius=3pt];
         \node at(3.5,0.68) {$\hat b_1$};
         \fill (3.3,0) circle[radius=3pt];
         \draw [] (2,1.6)node{$($} --(9.5,1.6)node{$)$};
         \node at(6,2.2) {$I_2$};
         \node at(9.6,0.5) {$b_2$};
         \fill (9.5,0) circle[radius=3pt];
         \node at(1.7,0.5) {$a_2$};
         \fill (2,0) circle[radius=3pt];
         \draw [] (2.7,-1.6)node{$($} --(9,-1.6)node{$)$};
         \node at(6,-2.2) {$\hat I_2$};
         \node at(2.6,0.6) {$\hat a_2$};
         \fill (2.7,0) circle[radius=3pt];
         \node at(8.8,0.6) {$\hat b_2$};
         \fill (9,0) circle[radius=3pt];
\end{tikzpicture}
\caption{The overlap of $I_1$ and $I_2$.}
 \label{fig:overlap}
\end{figure}

Since $\{\hat I_j\}_{j=1,2,3}$ is a cover of $S^1$, we have $(\hat I_1 \cup \hat I_2)' \subset \hat I_3$.
Therefore, if we set $\Xi_3(\gamma) := \gamma_2^{-1}\gamma_1^{-1}\gamma$, it is supported in $\hat I_3 \subset I_3$
and the map $\Xi_3$ is continuous because it is a composition of continuous maps. Moreover, $\gamma = \Xi_1(\gamma)\Xi_2(\gamma)\Xi_3(\gamma)$.

To show the last statements, let $\gamma \in \cU$ with $\supp \gamma \subset I_1$. Then, arguing as above, $\supp \Xi_2(\gamma) \subset (\hat b_1,b_1) \subset I_1 \cap I_2$, and $\supp \Xi_3(\gamma) \subset (a_1,\hat a_1)  \subset I_1 \cap I_3$.
Next, assume that $\supp \gamma \cap (a_2, b_1) = \emptyset$. In this case, it is straightforward that $\beta_1(\gamma) = 1$
and $\gamma_1(\theta) = \theta$ for $\theta \in (a_2, b_1)$ by \eqref{eq:gamma1} (by noting that $\gamma'(\theta) = 1$ for $\theta \in (a_2, b_1)$),
concluding the proof.
\end{proof}

\section{From DHR representations to local multiplier representations}\label{grouprp}
Let $(\cA, U, \Omega)$ be a conformal net, $\Gamma$ an infinite-dimensional Lie group, either $\Gamma = LG$ or $\Gamma = \overline \DiffSone$, $\tilde\Gamma$ its universal central extension,
$\tpi_0$ a (true) unitary representation of $\tilde\Gamma$, and $\rho$ a DHR representation of $\cA$ as in Section \ref{setting}.
Let $\cU \subset \Gamma$ be a neighbourhood of the identity element with continuous localizing maps as in Lemma \ref{lm:fragmentation-LG} or Lemma \ref{lm:fragmentation-diff}.
By taking a smaller neighbourhood if necessary (and call it again $\cU$), we may assume that
$\cU$ trivializes $\tilde\Gamma$, and we take a local continuous section $s : \cU \to \tilde\Gamma$,
where $\psi\circ s = \id$ on $\cU$ and $\psi: \tilde\Gamma \to \Gamma$ is the projection in the central extension.
The map $\pi_0 := \tpi_0\circ s$ can be seen as a local multiplier representation of $\Gamma$ with $\bbT$-valued  cocycle $c$, satisfying $\pi_0(\gamma)\pi_0(\gamma') = c(\gamma, \gamma')\pi_0(\gamma\gamma')$ for $\gamma, \gamma' \in \cU$, cf. Section 2.1.

For simplicity, we write $\gamma_j := \Xi_j(\gamma)$, for $\gamma\in\cU$ and $j=1,2,3$, where the neighbourhood $\cU$ of identity element $e\in\Gamma$ and the localizing maps $\Xi_j$ are as in Lemma \ref{lm:fragmentation-LG} or Lemma \ref{lm:fragmentation-diff}, subordinated to the cover $\{I_j\}_{j=1,2,3}$ of $S^1$.
Recall the unitary representation $\tpi_{\rho, I}$ of $\tilde \Gamma_I$ for each proper open interval $I \subset S^1$ defined in \eqref{eq:pilocal} from a given DHR representation $\rho$ of $\cA$.
This gives a local multiplier representation $\pi_{\rho, I}$ of $\Gamma_I$, such that $\pi_{\rho, I}=\tpi_{\rho, I}\circ s=\rho_I\circ\tpi_0\circ s=\rho_I\circ\pi_0$, using Section 2.1 and \eqref{eq:pilocal}.

We define a map $\pi_\rho$ on $\cU$ by
\begin{align}\label{eq:pidefonU}
 \pi_\rho(\gamma) := \pi_{\rho, I_1}(\gamma_1)\pi_{\rho, I_2}(\gamma_2)\pi_{\rho, I_3}(\gamma_3) c(\gamma_1,\gamma_2)^{-1}c(\gamma_1\gamma_2,\gamma_3)^{-1},
\end{align}
where we recall that $c$ does not depend on $I$, see Section \ref{setting}.
We need a preliminary lemma.
\begin{lemma}\label{lm:picompatible}
If $\gamma \in \cU$ has $\supp \gamma \subset I_1$, then $\pi_\rho(\gamma) = \pi_{\rho, I_1}(\gamma)$. Moreover, if $\supp \gamma \subset I_1\cap (I_2 \cup I_3)$, then $\pi_\rho(\gamma) = \pi_{\rho, I_1}(\gamma) = \pi_{\rho, I_2\cup I_3}(\gamma)$.
\end{lemma}

\begin{proof}
If $\supp \gamma \subset I_1$, then $\supp \gamma_2 \subset I_1 \cap I_2$ and $\supp \gamma_3 \subset I_1 \cap I_3$ by Lemma \ref{lm:fragmentation-LG} or Lemma \ref{lm:fragmentation-diff}, where both $I_1 \cap I_2$ and $I_1 \cap I_3$ are (proper connected) intervals.
Hence, since for every $I\subset \cU$ the section $s$ maps $\Gamma_I$ in $\tilde\Gamma_I$ and we assume that $\tpi_0(\tilde\Gamma_I)\subset\cA(I)$, we have $\pi_0(\gamma_{2}) \in \cA(I_1 \cap I_2)$ and $\pi_0(\gamma_{3}) \in \cA(I_1 \cap I_3)$, since $s(\Gamma_I)\subset\tilde\Gamma_I$ and $\tpi_0(\tilde\Gamma_I)\subset\cA(I)$ for all $I\in\cI$. By the compatibility condition of the DHR representation
\begin{align}\label{eq:compatiblei1}
 \pi_\rho(\gamma)
 &= \pi_{\rho, I_1}(\gamma_1)\pi_{\rho, I_2}(\gamma_2)\pi_{\rho, I_3}(\gamma_3)c(\gamma_1,\gamma_2)^{-1}c(\gamma_1\gamma_2,\gamma_3)^{-1} \nonumber \\
 &= \rho_{I_1}(\pi_0(\gamma_{1}))\rho_{I_2}(\pi_0(\gamma_{2}))\rho_{I_3}(\pi_0(\gamma_{3}))c(\gamma_1,\gamma_2)^{-1}c(\gamma_1\gamma_2,\gamma_3)^{-1} \nonumber \\ 
 &= \rho_{I_1}(\pi_0(\gamma_{1}))\rho_{I_1}(\pi_0(\gamma_{2}))\rho_{I_1}(\pi_0(\gamma_{3}))c(\gamma_1,\gamma_2)^{-1}c(\gamma_1\gamma_2,\gamma_3)^{-1} \nonumber \\  
 &= \pi_{\rho, I_1}(\gamma_1)\pi_{\rho, I_1}(\gamma_2)\pi_{\rho, I_1}(\gamma_3)c(\gamma_1,\gamma_2)^{-1}c(\gamma_1\gamma_2,\gamma_3)^{-1} \nonumber \\
 &= \pi_{\rho, I_1}(\gamma),
\end{align}
since $ \pi_{\rho, I_1}$ is a local multiplier representation of $\Gamma_{I_1}$. 
If, moreover, $\supp \gamma$ is contained in the disjoint union $(I_1\cap I_2) \cup (I_1\cap I_3)$, 
then $\gamma_1$ in the fragmentation of $\gamma$ is actually a product of two elements $\gamma_1 = \gamma_{1,2}\gamma_{1,3}$, where $\supp \gamma_{1,j} \subset I_1 \cap I_j$, $j=2,3$.
Again by the DHR compatibility condition
\begin{align*}
 \pi_{\rho, I_1}(\gamma_{1,j}) &= \rho_{I_1}(\pi_0(\gamma_{1,j})) = \rho_{I_1\cap I_j}(\pi_0(\gamma_{1,j})) = \rho_{I_j}(\pi_0(\gamma_{1,j})) \\
 &= \rho_{I_2\cup I_3}(\pi_0(\gamma_{1,j})) =\pi_{\rho, I_2\cup I_3}(\gamma_{1,j}).
\end{align*}
Therefore, 
\begin{align*}
  \pi_{\rho, I_1}(\gamma_1) &= \pi_{\rho, I_1}(\gamma_{1,2})\pi_{\rho, I_1}(\gamma_{1,3}) c(\gamma_{1,2}, \gamma_{1,3}) = \pi_{\rho, I_2\cup I_3}(\gamma_{1,2})\pi_{\rho, I_2\cup I_3}(\gamma_{1,3}) c(\gamma_{1,2}, \gamma_{1,3}) \\
  &= \pi_{\rho, I_2\cup I_3}(\gamma_{1,2}\gamma_{1,3})
  = \pi_{\rho, I_2\cup I_3}(\gamma_1),
\end{align*}
and hence
\begin{align*}
 \pi_\rho(\gamma) &= \pi_{\rho, I_1}(\gamma_1)\pi_{\rho, I_2}(\gamma_2)\pi_{\rho, I_3}(\gamma_3) c(\gamma_1,\gamma_2)^{-1}c(\gamma_1\gamma_2,\gamma_3)^{-1} \\
 &= \pi_{\rho, I_2\cup I_3}(\gamma_1)\pi_{\rho, I_2}(\gamma_2)\pi_{\rho, I_3}(\gamma_3) c(\gamma_1,\gamma_2)^{-1}c(\gamma_1\gamma_2,\gamma_3)^{-1} \\
 &= \pi_{\rho, I_2\cup I_3}(\gamma),
\end{align*}
where we used the fact that $ \pi_{\rho, I_j}(\gamma_j) = \rho_{I_j}(\pi_0(\gamma_j)) = \rho_{I_2\cup I_3}(\pi_0(\gamma_j))=\pi_{\rho, I_2\cup I_3}(\gamma_j)$ for $j=2,3$. 
\end{proof}

The key idea to construct a new representation of $\Gamma$
is that the local multiplier representations $\{\pi_{\rho,I}\}$ of the subgroups $\{\Gamma_I\}$ of elements localized in $I$
defined through a DHR representation $\rho$ can be glued together, see \eqref{eq:pidefonU}.
Yet, there is a concern that the cocycle does not match on the whole circle $S^1$ and $\pi_\rho$ is no longer a local multiplier representation.
We will exclude this possibility by exploiting the fact that the cocycle is completely determined by localized elements and
by using the localized unitary equivalence between $\pi_{\rho,I}$ and the vacuum representation $\pi_0$.

\begin{theorem}\label{th:local}
 There is a sufficiently small neighbourhood $\cV \subset \cU$ such that
 the map $\pi_\rho$ defined in \eqref{eq:pidefonU} restricted to $\cV$ does not depend on the decomposition $\gamma = \gamma_1 \gamma_2 \gamma_3$, where $\gamma_j \in \cV$,
 and $\gamma_j$ has support in $I_j$, $j=1,2,3$.
 Moreover, the map $\pi_\rho$ is a local multiplier representation of $\Gamma$, namely, for $\gamma, \gamma' \in \cV$,
 \[
 \pi_\rho(\gamma)\pi_\rho(\gamma') = c(\gamma,\gamma')\pi_\rho(\gamma\gamma'),
 \]
 where $c$ is the same cocycle of the local multiplier representation $\pi_0$.
\end{theorem}

In the following proof, ``local'' refers to neighbourhoods of the group $\Gamma$, while ``localized''
 is about the supports of group elements.

\begin{proof} Let $\cV$ be a neighbourhood of $e_\Gamma$ such that there is a large enough $n \in \bbN$
such that $\cV^n := \underset{n\text{-times}}{\cV \cdots \cV} \subset \cU$.
 Let $\gamma = \gamma'_1 \gamma'_2 \gamma'_3$ be another decomposition, where again $\gamma'_j \in \cV$ and $\gamma'_j$ has support in $I_j$. As it holds that $\gamma_3^{-1}\gamma_2^{-1}\gamma_1^{-1}\gamma'_1\gamma'_2\gamma'_3 = e_\Gamma$ in $\Gamma$
and $\pi_0(\gamma_1^{-1}) = c(\gamma_1,\gamma_1^{-1})\pi_0(\gamma_1)^*$,
we set
\[
 c(\gamma_1,\gamma_2,\gamma_3,\gamma'_1,\gamma'_2,\gamma'_3) := \pi_0(\gamma_3)^*\pi_0(\gamma_2)^*\pi_0(\gamma_1^{-1}\gamma'_1)\pi_0(\gamma'_2)\pi_0(\gamma'_3). 
\]
Using the assumption (see Section \ref{setting}) that $\tpi(\phi(H)) \subset \bbT\bb1$,
where $\phi, H$ are as in Section \ref{central}, we see that $c(\gamma_1,\gamma_2,\gamma_3,\gamma'_1,\gamma'_2,\gamma'_3)\in \bbT$. On the other hand, in the vacuum representation $\pi_0$ we have
\begin{align*}
 \pi_0(\gamma) &= \pi_0(\gamma_1)\pi_0(\gamma_2)\pi_0(\gamma_3)c(\gamma_1,\gamma_2)^{-1}c(\gamma_1\gamma_2,\gamma_3)^{-1} \\
 &= \pi_0(\gamma'_1)\pi_0(\gamma'_2)\pi_0(\gamma'_3)c(\gamma'_1,\gamma'_2)^{-1}c(\gamma'_1\gamma_2',\gamma'_3)^{-1},
\end{align*}
where the local multiplier representation $\pi_0$ is defined in the beginning of this section. Therefore, the following relation involving the cocycle $c$ from the vacuum representation $\pi_0$ holds
 \begin{align}\label{eq:csix}
 &c(\gamma_1,\gamma_2,\gamma_3,\gamma'_1,\gamma'_2,\gamma'_3)c(\gamma_1^{-1},\gamma'_1)c(\gamma_1,\gamma_1^{-1})^{-1}c(\gamma_1,\gamma_2)c(\gamma_1\gamma_2,\gamma_3)c(\gamma'_1,\gamma'_2)^{-1}c(\gamma'_1\gamma_2',\gamma'_3)^{-1} \nonumber \\
 &= 1,
 \end{align}
where $c(\gamma_1,\gamma_2,\gamma_3,\gamma'_1,\gamma'_2,\gamma'_3) \in \bbT$ is defined above.

Now we take two decompositions $\gamma = \gamma_1\gamma_2\gamma_3 = \gamma'_1\gamma'_2\gamma'_3$.
Note that $\pi_{\rho, I}$ is unitarily equivalent to $\pi_{0}|_{\Gamma_I}$ on any proper interval $I$
as $\rho_I$ is spacial (see Section \ref{sec:DHR}), therefore,
\begin{align}\label{eq:gammainv}
 \pi_{\rho, I_1}(\gamma_1)^*\pi_{\rho, I_1}(\gamma'_1) = c(\gamma_1^{-1},\gamma'_1)c(\gamma_1,\gamma_1^{-1})^{-1}\pi_{\rho, I_1}(\gamma_1^{-1}\gamma'_1),
\end{align}
and that $\gamma_1^{-1}\gamma'_1 = \gamma_2\gamma_3\gamma_3^{\prime-1}\gamma_2^{\prime-1}$ has support in $I_1 \cap (I_2\cup I_3)$.
Then we can again use the unitary equivalence between $\pi_\rho$ and $\pi_0$ on $I_2\cup I_3$ and Lemma \ref{lm:picompatible}
to obtain
\begin{align*}
 &\pi_{\rho, I_3}(\gamma_3)^*\pi_{\rho, I_2}(\gamma_2)^*\pi_{\rho, I_1}(\gamma_1^{-1}\gamma'_1)\pi_{\rho, I_2}(\gamma'_2)\pi_{\rho, I_3}(\gamma'_3) \\
 &=\pi_{\rho, I_2\cup I_3}(\gamma_3)^*\pi_{\rho, I_2\cup I_3}(\gamma_2)^*\pi_{\rho, I_2\cup I_3}(\gamma_1^{-1}\gamma'_1)\pi_{\rho, I_2\cup I_3}(\gamma'_2)\pi_{\rho, I_2\cup I_3}(\gamma'_3) \\
 &= c(\gamma_1,\gamma_2,\gamma_3,\gamma'_1,\gamma'_2,\gamma'_3).
\end{align*}
Using \eqref{eq:csix} and \eqref{eq:gammainv},
the above computations lead  to the equality
\begin{align*}
 &\pi_{\rho, I_1}(\gamma_1)\pi_{\rho, I_2}(\gamma_2)\pi_{\rho, I_3}(\gamma_3)c(\gamma_1,\gamma_2)^{-1}c(\gamma_1\gamma_2,\gamma_3)^{-1} \\
 =\;& \pi_{\rho, I_1}(\gamma'_1)\pi_{\rho, I_2}(\gamma'_2)\pi_{\rho, I_3}(\gamma'_3)c(\gamma'_1,\gamma'_2)^{-1}c(\gamma'_1\gamma_2',\gamma'_3)^{-1}.
\end{align*}
In other words, $\pi_\rho$ is well-defined on $\cV$.

Next we show that $\pi_\rho$ is a local multiplier representation of $\Gamma$ in $\cV$.
Let $\gamma, \gamma'\in\cV$ and we take decompositions $\gamma=\gamma_1\gamma_2\gamma_3, \gamma'=\gamma'_1\gamma'_2\gamma'_3$.
We first look at the product $\gamma_3\gamma'_1$. This is supported in $I_1\cup I_3$.
We can find another decomposition $\gamma_3\gamma'_1 = \gamma''_1\gamma''_3$, where $\gamma''_j$ is supported in $I_j$. Indeed, using the map $\Xi_1$ from Section \ref{fragmentation} and by shrinking $I_2$ if necessary such that $I_2 \cap \supp \gamma'_1 = \emptyset$ we can define $\gamma''_1 = \Xi_1(\gamma_3\gamma'_1)$ supported in $I_1$. We now define $\gamma''_3 = (\gamma''_1)^{-1}\gamma_3\gamma'_1$, and we show that $\gamma''_3$ is supported in $I_3$.
We treat the two cases separately:
\begin{itemize}
	\item The case of $\Gamma=LG$: following Lemma \ref{lm:fragmentation-LG}, $\Xi_1(\gamma_3\gamma'_1)=\mathrm{Exp}(\chi_1\eta)$, where $\chi_1=1$ on a slightly larger set than $S^1\setminus(I_2\cup I_3)$ and $\eta\in L\frg$ such that $\mathrm{Exp}(\eta)=\gamma_3\gamma'_1$. Hence, $\gamma_3\gamma'_1(\theta)=\Xi_1(\gamma_3\gamma'_1)(\theta)$ implies $\gamma''_3(\theta)=e_{G}$, which holds on $S^1 \setminus (I_2\cup I_3)$. For $\theta\in I_2\setminus I_3$, one has $\gamma_3(\theta)\gamma'_1(\theta)=e_{G}$ since $\gamma_3$ is supported in $I_3$ and $\gamma'_1$ is supported in $I_1\setminus I_2$. Altogether, $\gamma''_3(\theta)=e_G$ on $S^1 \setminus I_3$.
	\item The case of $\Gamma=\uDiffSone$: following a similar argument as in the case of $\Gamma=LG$, but using Lemma \ref{lm:fragmentation-diff}, we again have that $\gamma''_3(\theta)=\theta$ where $\gamma_3\gamma'_1(\theta)=\Xi_1(\gamma_3\gamma'_1)(\theta)$.
	We know that $\gamma_3\gamma'_1(\theta)=\Xi_1(\gamma_3\gamma'_1)(\theta)$ on $[\hat a_1,\hat b_1] = \hat I_1$, hence $\supp \gamma''_3 \subset S^1 \setminus \hat I_1$. Moreover, in $(a_2,b_1) \subset I_2$ we have $\gamma'_1(\theta)=\theta$ (since the support of $\gamma'_1$ is disjoint from $I_2$ by assumption) and $\gamma_3(\theta)=\theta$ (since $(a_2,b_1)$ is disjoint from $(a_3,b_3)$), thus $\gamma_3\gamma'_1(\theta)=\theta$
	for $\theta \in (a_2, b_1)$ and by the last part of Lemma \ref{lm:fragmentation-diff} $\Xi_1(\gamma_3\gamma'_1)(\theta)=\theta$ as well,
	showing that $\gamma''_3(\theta) = \theta$ for $\theta \in (a_2, b_1)$.
	Altogether, $\supp \gamma''_3 \subset (I_1 \cup I_3) \setminus (\hat I_1 \cup (a_2, b_1)) \subset I_3$.
\end{itemize}
By repeating the above argument to $\gamma_2\gamma_1''$, by shrinking $I_3$ if necessary such that $I_3 \cap \supp \gamma''_1 = \emptyset$, we obtain a decomposition $\gamma'''_1\gamma'''_2$ where $\gamma'''_1$, $\gamma'''_2$ are supported respectively in $I_1$ and $I_2$. Swapping the role of $I_1$ and $I_2$, we can construct a new decomposition for $\gamma_3''\gamma_2'$ following a similar argument. In conclusion, we have
\begin{align*}
 \gamma\gamma' &= \gamma_1\gamma_2\gamma_3\gamma'_1\gamma'_2\gamma'_3 \\
 &= \gamma_1\gamma_2\gamma''_1\gamma''_3\gamma'_2\gamma'_3 \\
 &= \gamma_1\gamma'''_1\gamma'''_2\gamma''''_2\gamma''''_3\gamma'_3,
\end{align*}
where $\gamma'''_j,\gamma''''_j$ are supported in $I_j$.

Again, there are $c(\gamma_3,\gamma'_1,\gamma''_1,\gamma''_3),
c(\gamma_2,\gamma''_1,\gamma'''_1,\gamma'''_2), c(\gamma''_3,\gamma'_2,\gamma''''_2,\gamma''''_3)\in\bbT$ that satisfy the following relations
in the vacuum representation
\begin{align*}
 \pi_0(\gamma_3)\pi_0(\gamma'_1) &= \pi_0(\gamma''_1)\pi_0(\gamma''_3)c(\gamma_3,\gamma'_1,\gamma''_1,\gamma''_3),\\
 \pi_0(\gamma_2)\pi_0(\gamma''_1) &= \pi_0(\gamma'''_1)\pi_0(\gamma'''_2)c(\gamma_2,\gamma''_1,\gamma'''_1,\gamma'''_2),\\
 \pi_0(\gamma''_3)\pi_0(\gamma'_2) &= \pi_0(\gamma''''_2)\pi_0(\gamma''''_3)c(\gamma''_3,\gamma'_2,\gamma''''_2,\gamma''''_3).
\end{align*}
Therefore,
\begin{align*}
 &c(\gamma,\gamma') \pi_0(\gamma\gamma')\\
 =&\;\pi_0(\gamma)\pi_0(\gamma')\\
 =&\;c(\gamma_1,\gamma_2)^{-1}c(\gamma_1\gamma_2,\gamma_3)^{-1}c(\gamma'_1,\gamma'_2)^{-1}c(\gamma'_1\gamma'_2,\gamma'_3)^{-1}\\
 &\times \pi_0(\gamma_1)\pi_0(\gamma_2)\pi_0(\gamma_3)\pi_0(\gamma'_1)\pi_0(\gamma'_2)\pi_0(\gamma'_3) \\
 =&\;c(\gamma_1,\gamma_2)^{-1}c(\gamma_1\gamma_2,\gamma_3)^{-1}c(\gamma'_1,\gamma'_2)^{-1}c(\gamma'_1\gamma'_2,\gamma'_3)^{-1}\\
 &\times c(\gamma_3,\gamma'_1,\gamma''_1,\gamma''_3)c(\gamma_2,\gamma''_1,\gamma'''_1,\gamma'''_2) c(\gamma''_3,\gamma'_2,\gamma''''_2,\gamma''''_3)\\
 &\times \pi_0(\gamma_1)\pi_0(\gamma'''_1)\pi_0(\gamma'''_2)\pi_0(\gamma''''_2)\pi_0(\gamma''''_3)\pi_0(\gamma'_3)\\
 =&\;c(\gamma_1,\gamma_2)^{-1}c(\gamma_1\gamma_2,\gamma_3)^{-1}c(\gamma'_1,\gamma'_2)^{-1}c(\gamma'_1\gamma'_2,\gamma'_3)^{-1}\\
 &\times c(\gamma_3,\gamma'_1,\gamma''_1,\gamma''_3)c(\gamma_2,\gamma''_1,\gamma'''_1,\gamma'''_2) c(\gamma''_3,\gamma'_2,\gamma''''_2,\gamma''''_3)\\
 &\times c(\gamma_1,\gamma'''_1)c(\gamma'''_2,\gamma''''_2)c(\gamma''''_3,\gamma'_3)\cdot\pi_0(\gamma_1\gamma'''_1)\pi_0(\gamma'''_2\gamma''''_2)\pi_0(\gamma''''_3\gamma'_3)\\
 =&\;c(\gamma_1,\gamma_2)^{-1}c(\gamma_1\gamma_2,\gamma_3)^{-1}c(\gamma'_1,\gamma'_2)^{-1}c(\gamma'_1\gamma'_2,\gamma'_3)^{-1}\\
 &\times c(\gamma_3,\gamma'_1,\gamma''_1,\gamma''_3)c(\gamma_2,\gamma''_1,\gamma'''_1,\gamma'''_2) c(\gamma''_3,\gamma'_2,\gamma''''_2,\gamma''''_3)\\
 &\times c(\gamma_1,\gamma'''_1)c(\gamma'''_2,\gamma''''_2)c(\gamma''''_3,\gamma'_3) c(\gamma_1\gamma'''_1,\gamma'''_2\gamma''''_2)c(\gamma_1\gamma'''_1\gamma'''_2\gamma''''_2,\gamma''''_3\gamma'_3)\cdot\pi_0(\gamma\gamma'),
 \end{align*}
or equivalently, the following relation between scalars holds
\begin{align*}
 c(\gamma,\gamma') =&\;c(\gamma_1,\gamma_2)^{-1}c(\gamma_1\gamma_2,\gamma_3)^{-1}c(\gamma'_1,\gamma'_2)^{-1}c(\gamma'_1\gamma'_2,\gamma'_3)^{-1}\\
 &\times c(\gamma_3,\gamma'_1,\gamma''_1,\gamma''_3)c(\gamma_2,\gamma''_1,\gamma'''_1,\gamma'''_2) c(\gamma''_3,\gamma'_2,\gamma''''_2,\gamma''''_3)\\
 &\times c(\gamma_1,\gamma'''_1)c(\gamma'''_2,\gamma''''_2)c(\gamma''''_3,\gamma'_3)\cdot c(\gamma_1\gamma'''_1,\gamma'''_2\gamma''''_2)c(\gamma_1\gamma'''_1\gamma'''_2\gamma''''_2,\gamma''''_3\gamma'_3).
 \end{align*}
Now, in order to show that $\pi_\rho$ is a local multiplier representation, we only have to compute
\begin{align*}
 &\;\pi_\rho(\gamma)\pi_\rho(\gamma') \\
 =&\;c(\gamma_1,\gamma_2)^{-1}c(\gamma_1\gamma_2,\gamma_3)^{-1}c(\gamma'_1,\gamma'_2)^{-1}c(\gamma'_1\gamma'_2,\gamma'_3)^{-1}\\
 &\times \pi_\rho(\gamma_1)\pi_\rho(\gamma_2)\pi_\rho(\gamma_3)\pi_\rho(\gamma'_1)\pi_\rho(\gamma'_2)\pi_\rho(\gamma'_3) \\
 =&\;c(\gamma_1,\gamma_2)^{-1}c(\gamma_1\gamma_2,\gamma_3)^{-1}c(\gamma'_1,\gamma'_2)^{-1}c(\gamma'_1\gamma'_2,\gamma'_3)^{-1}\\
 &\times \pi_\rho(\gamma_1)\pi_\rho(\gamma'''_1)\pi_\rho(\gamma'''_2)\pi_\rho(\gamma''''_2)\pi_\rho(\gamma''''_3)\pi_\rho(\gamma'_3)\\
 &\times c(\gamma_3,\gamma'_1,\gamma''_1,\gamma''_3)c(\gamma_2,\gamma''_1,\gamma'''_1,\gamma'''_2) c(\gamma''_3,\gamma'_2,\gamma''''_2,\gamma''''_3)\\
 =&\;c(\gamma,\gamma')\left(c(\gamma_1,\gamma'''_1)c(\gamma'''_2,\gamma''''_2)c(\gamma''''_3,\gamma'_3)\cdot c(\gamma_1\gamma'''_1,\gamma'''_2\gamma''''_2)c(\gamma_1\gamma'''_1\gamma'''_2\gamma''''_2,\gamma''''_3\gamma'_3)\right)^{-1}\\
 &\times \pi_\rho(\gamma_1)\pi_\rho(\gamma'''_1)\pi_\rho(\gamma'''_2)\pi_\rho(\gamma''''_2)\pi_\rho(\gamma''''_3)\pi_\rho(\gamma'_3)\\
 =&\;c(\gamma,\gamma')\left( c(\gamma_1\gamma'''_1,\gamma'''_2\gamma''''_2)c(\gamma_1\gamma'''_1\gamma'''_2\gamma''''_2,\gamma''''_3\gamma'_3)\right)^{-1}\\
 &\times \pi_\rho(\gamma_1\gamma'''_1)\pi_\rho(\gamma'''_2\gamma''''_2)\pi_\rho(\gamma''''_3\gamma'_3)\\ 
 =&\;c(\gamma,\gamma')\pi_\rho(\gamma\gamma')
\end{align*}
where we used local equivalence between $\pi_\rho$ and $\pi_0$ in the 2nd (on $I_3\cup I_1$, $I_2\cup I_1$, and $I_3\cup I_2$) and 4th equalities (on $I_1$, $I_2$, and $I_3$), the previous scalar relation in the 3rd equality, and the well-definedness of $\pi_\rho$ (independence on the partition of an element of $\Gamma$
into localized elements) in the 5th equality.
\end{proof}

K\"oster states that a DHR representation induces a projective representation of $\DiffSone$ \cite[after Proposition III.2]{KoesterThesis}.
In comparison, here we construct a local multiplier representation in the case $\Gamma = LG, \DiffSone$, with the same cocycle as that of the vacuum representation.

\section{Some applications}\label{application}
\subsection{Group representations from DHR sectors}\label{grouprep}
Let $(\cA, U, \Omega)$, $\Gamma$ (either $LG$ or $\uDiffSone$), $\tpi_0$, $\rho$ be as in Section \ref{setting}.
In Section \ref{grouprp}, we obtained a local multiplier representation $\pi_\rho$ of $\Gamma$ on the Hilbert space $\cH_\rho$ of the DHR representation $\rho$.
Furthermore, denoted by $\tilde U$ the extension of $U$ (projective representation of $\Diff_+(S^1)$) to $\Vir$, we have that $\tilde U(\Vir_I) \subset \cA(I)$, and that the lift of the $2\pi$-rotation is a scalar (see Remark \ref{rm:UinConformalNet}). Therefore, given any conformal net $(\cA, U, \Omega)$ and one of its DHR representations $\rho$, the construction in Section \ref{grouprp} can be applied to the extension $\tilde U$ of $U$ to $\Vir$, giving a local multiplier representation $U_\rho$ of $\Gamma=\uDiffSone$. In the case $\Gamma=LG$, the local multiplier representation $U_\rho$ of $\uDiffSone$ makes $\pi_\rho$ covariant, since we assumed that $U$ makes $\tpi_0$ covariant. Note that for a generic conformal net $(\cA, U, \Omega)$, we assumed $\tilde U = \tpi_0$.

\begin{theorem}\label{th:ext}
 The map $\pi_\rho$ defined in \eqref{eq:pidefonU} extends to a unitary positive-energy representation of $\tilde\Gamma$ (either $\tilde LG$ or $\Vir$)
 denoted by $\tpi_\rho$.
 Furthermore, if this construction is applied to $U_\rho$ above in place of $\pi_\rho$
 and if $\rho$ is factorial\footnote{$\rho$ is factorial if $\left(\bigcup_{I \in \cI} \rho_I(\cA(I))\right)''$ is a factor, i.e.\ a von Neumann algebra with trivial center.},
 then $U_\rho$ induces a projective representation of $\DiffSone$, i.e.,  the lift of the $2\pi$-rotation is scalar.
\end{theorem}

\begin{proof}
 As $\tilde\Gamma$ is connected, simply connected and locally path connected, $\pi_\rho$ in Theorem \ref{th:local} extends to an everywhere defined unitary representation $\tpi_\rho$
 of $\tilde\Gamma$ by Lemma \ref{lm:extlocal}.
 Similarly, denote by $\tilde U_\rho$ the unitary representation of $\Vir$ extending the local multiplier representation $U_\rho$ of $\uDiffSone$ constructed in Theorem \ref{th:local}.
 In the latter case, since $\Vir = \bbR \times_{\bfB} \uDiffSone$ has a global section,
 $U_\rho$ can be extended everywhere in $\uDiffSone$ as a global multiplier representation, again denoted by $U_\rho$.

 Weiner showed that any DHR representation $\rho$ of a conformal (diffeomorphism covariant) net has positive energy, cf. \cite[Proposition 3.8]{Weiner06}.
 Note that the implementation of diffeomorphisms as a projective representation used in \cite[Proposition 3.3]{Weiner06} is given by \cite{DFK04},
 whereas our construction applied to $U$ seen as local multiplier representation of $\Gamma=\uDiffSone$ fixes the phase in Weiner's implementation. Thus, by our implementation of the M\"obius group sitting in $\Vir$, the representation $\tilde U_\rho$ of $\Vir$ becomes a positive-energy representation.
 \begin{itemize}
  \item If $\Gamma = \uDiffSone$, this gives the positivity of energy, where we take $\tpi_\rho$ as $\tilde U_\rho$.
  \item If $\Gamma = LG$, we see that $\tilde U_\rho$ satisfies the covariance relation with $\tpi_\rho$ in any interval $I$, since $\tpi_0$ is coviariant with respect to $U$. Thus it satisfies covariance in $\tilde LG$. Then the representation $(\tpi_\rho, \tilde U_\rho)$ of $\tilde LG \rtimes \Vir$
  (restricted to $\tilde LG \rtimes \bbR$) has positive energy.
 \end{itemize}

 Assume that $\rho$ is factorial. We denote by $\tau_t \in \Diff_+(S^1)$ the rotation by $t \in S^1$ and by $\tau_{\bar t}$ the lift to $\uDiffSone$ of the rotation by $\bar t \in \bbR$
 (we are first given $\bar t \in \bbR$ and $t$ is its quotient by $2\pi\bbZ$). 
 Recall that for $\bar t \in \bbR$ and $t = \bar t /2\pi\bbZ$, it holds that
 $\cAd U_\rho(\tau_{\bar t}) (\rho_{I}(x)) = \rho_{\tau_t I}(\cAd U(\tau_t)(x))$ for any $x \in \cA(I), I \in \cI$.
 If $\bar t$ is a lift of the $2\pi$-rotation, then it holds that $\cAd U_\rho(\tau_{\bar t})(\rho_{I}(x)) = \rho_{I}(\cAd U(\tau_t)(x)) = \rho_{I}(x)$.
 This implies that $U_\rho(\tau_{\bar t}) \in \left(\bigcup_{I \in \cI} \rho_I(\cA(I))\right)'$.
 As one can write $\tau_{\bar t}$ as a product of diffeomorphisms localized in intervals with disjoint closures, one has 
 $U_\rho(\tau_{\bar t}) \in \left(\bigcup_{I \in \cI} \rho_I(\cA(I))\right)''$ too.
 As $\rho$ is factorial, this implies that $U_\rho(\tau_{\bar t}) \in \bbT\bb1$ and we can see $U_\rho$ 
 as a projective representation of $\Diff_+(S^1)$.
\end{proof}

Zellner showed that, assuming the (unpublished) results of \cite{Gloeckner16Measurable}
(see \cite[Reference 6]{Zellner16}),
any positive-energy representation of $\tilde LG$ is smooth \cite[Theorem 2.16]{Zellner16}.
Let $\Gamma = LG$, $\tpi_0 = \tpi_{G,k}$ the vacuum representation of $\tilde LG$ at level $k$. Let $\rho$ be one of the DHR representations of the loop group net.
Assuming the result of Zellner, we can show that the representation $\tpi_\rho$ of $\tilde LG$ constructed above has the same level $k$.
Indeed, as $\tpi_{\rho}$ has positive energy by Theorem \ref{th:ext} and $\tpi_\rho$ is smooth, it is the direct sum of irreducible representations with the same level, constructed and classified in \cite{PS86}.
The cocycle coincides with $c$ of the vacuum representation $\tpi_{G,k}$, therefore, it has level $k$ as well.
This result has also been proved by Carpi and Weiner by different methods, and is to appear in \cite{CW2X}.

In \cite{Henriques19LoopGroups}, for $\cA = \cA_{G,k}$, $\Gamma = LG$, and a DHR representation $\rho$ of $\cA_{G,k}$,
a corresponding level $k$ positive-energy representation of $\tilde LG$ has been constructed. Here we extended the result to include also chiral extensions of the loop group nets $\cA_{G,k}$, using the argument of local representations instead of colimits.

\subsection{Conformal covariance of DHR representations}\label{covariance}

Let $\cA$ be any conformal net as in Section \ref{sec:conformalnet}, $\Gamma = \uDiffSone$, and $\rho = \{\rho_I\}$ be a DHR representation of $\cA$ as in Section \ref{sec:DHR}. 
We say that $\rho$ is \textbf{diffeomorphism covariant} if there is a projective representation $W_\rho$ of $\DiffSone$ on $\cH_\rho$ (the representation space of $\rho$) such that
for any $I \in \cI$ and $x \in \cA(I)$ it holds that $\rho_{\gamma I}(\cAd U(\gamma)(x)) = \cAd W_\rho(\gamma)(\rho_I(x))$.

It turns out that diffeomorphism covariance is automatic for factorial $\rho$ (consequently, if $\rho$ is a direct sum of
irreducible representations). By Theorem \ref{th:ext}, we can choose $W_\rho$ to be
the projective representation of $\DiffSone$ on $\cH_\rho$ induced by $U_\rho$.
Indeed, let $x \in \cA(I)$ for some $I \in \cI$.
Let $\cU$ be a small neighbourhood of $\DiffSone
$ as in Lemma \ref{lm:fragmentation-diff}.
There is a smaller neighbourhood $\cV$ of $\cU$ such that any $\gamma \in \cV$
can be written as $\gamma = \gamma_1\gamma_2$, where $\gamma_j$ are supported in $I_j \in \cI$, $j=1,2$,
where $I_1 \cup I_2 = S^1$, $I \subset I_1$, $\gamma I \subset I_1$ and $\gamma_j \in \cU$ (cf.\ Lemma \ref{lm:fragmentation-diff}).
In this situation, $\overline{I_2} \cap \overline I = \emptyset$ and
we have (the cocycle appears with its conjugate thus it cancels)
\begin{align*}
 \rho_{\gamma I}(\cAd U(\gamma)(x))
 &= \rho_{I_1}(\cAd U(\gamma)(x)) \\
 &= \rho_{I_1}(U(\gamma)xU(\gamma)^*) = \rho_{I_1}(U(\gamma_1)U(\gamma_2)xU(\gamma_2)^*U(\gamma_1)^*) \\
 &= \rho_{I_1}(U(\gamma_1)xU(\gamma_1)^*) = U_\rho(\gamma_1)\rho_I(x) U_\rho(\gamma_1)^* \\
 &=  U_\rho(\gamma_1) U_\rho(\gamma_2)\rho_I(x) U_\rho(\gamma_2)^* U_\rho(\gamma_1)^* \\
 &= \cAd U_\rho(\gamma)(\rho_I(x)).
\end{align*}
This proves the desired covariance relation for $\rho$.

In \cite{Gui21Categorical}, Gui defines and proves ``conformal covariance'' of a DHR sector $\rho$ ($\pi_i$ in his notation)
as the existence of a representation
$\tilde W_\rho$ of a central extension $\cG_\cA$ of $\uDiffSone$ such that
$\tilde W_\rho(\gamma) = \rho_I(\tilde U(\gamma))$ for $\gamma$ in the group $\cG_\cA(I)$ (\cite[Theorem 2.2]{Gui21Categorical}).
This property follows also from our construction, and indeed we can take $\tilde W_\rho = \tilde U_\rho$,
as $\cG_\cA$ is a certain quotient group of the universal central extension $\Vir$.

\subsection{Naturality of covariance cocycles}

Let $\rho = \{\rho_I\},\sigma = \{\sigma_I\}$ be DHR representations of a conformal net $\cA$ as in Section \ref{sec:DHR}.
Without loss of generality, assume that they are ``localized'' respectively in $I_1, I_2\in\cI$, i.e.\ $\rho_{I_1'} = \rho_0$ and $\sigma_{I_2'} = \rho_0$, where $\rho_0$ is the vacuum representation of $\cA$ (in particular, the representation spaces $\cH_\rho$ and $\cH_\sigma$ are both $\cH_0$).
In particular, $\rho_{I_1}, \sigma_{I_2}$ are endomorphisms of $\cA(I_1), \cA(I_2)$ respectively, as in Section \ref{sec:DHR}.
Assume further that $I_1 \cup I_2 \subset I$ for some proper interval $I\in\cI$. Then, both $\rho_I$ and $\sigma_I$ are endomorphisms of $\cA(I)$ and every intertwining operator $V$ between the DHR representations $\rho$ and $\sigma$ (i.e.\ $V \in \cB(\cH_0)$ such that $V \rho_J(x) = \sigma_J(x) V$ for all $x \in \cA(J)$ and $J \in \cI$) belongs to $\cA(I)$. Indeed, $V x = V \rho_{I'}(x) = \sigma_{I'}(x) V = x V$ for every $x\in\cA(I')$, hence $V \in \cA(I')' = \cA(I)$ by Haag duality.

For a DHR endomorphism $\sigma$ and $\gamma \in \uDiffSone$, we define the \textbf{covariance cocycle} $z_{\sigma}(\gamma) := U(\gamma)U_\sigma(\gamma)^*$, where $U$, $U_\sigma$ are global multiplier representations of $\uDiffSone$ on $\cH_0$ from Theorem \ref{th:ext}.
By definition, $z_{\sigma}(\gamma)$ is a unitary intertwiner between $\sigma$ and $\sigma^\gamma := \cAd U(\gamma) \circ \sigma \circ \cAd U(\gamma^{-1})$.
Note that if $\sigma$ is localized in $I\in\cI$, then $\sigma^\gamma$ is localized in $\gamma I \in \cI$.

As a consequence of the fact that $U$ and $U_\sigma$ have the same cocycle $c$ by our construction, the unitaries $z_\sigma(\gamma)$ satisfy the \textbf{covariance cocycle identity}:
\begin{align*}
 z_\sigma(\gamma_1\gamma_2)
 &= U(\gamma_1\gamma_2)U_\sigma(\gamma_1\gamma_2)^* \\
 &= c(\gamma_1,\gamma_2)^{-1}U(\gamma_1)U(\gamma_2)c(\gamma_1,\gamma_2)U_\sigma(\gamma_2)^*U_\sigma(\gamma_1)^*\\
 &= U(\gamma_1)U(\gamma_2)U_\sigma(\gamma_2)^*U_\sigma(\gamma_1)^* \\
 &= U(\gamma_1)U(\gamma_2)U_\sigma(\gamma_2)^* U(\gamma_1)^* U(\gamma_1)U_\sigma(\gamma_1)^* \\
 &=\cAd U(\gamma_1)(z_\sigma(\gamma_2))\cdot z_\sigma(\gamma_1),
\end{align*}
where we used the equality $U_\sigma(\gamma_1\gamma_2)^*=c(\gamma_1,\gamma_2)U_\sigma(\gamma_2)^*U_\sigma(\gamma_1)^*$ and $c$ is the same cocycle of $U$. 
Another consequence of our result for $\Gamma = \uDiffSone$ is that the covariance cocycles satisfy the following \textbf{naturality property}:
$z_{\sigma}(\gamma) V = V^\gamma z_{\rho}(\gamma)$ for every $\gamma\in\uDiffSone$ 
and $V$ intertwiner between $\rho$ and $\sigma$, where $V^\gamma := \cAd U(\gamma)(V)$ is an intertwiner between $\rho^\gamma$ and $\sigma^\gamma$.
Indeed, if $\gamma \in \Diff_+(I_1)$, we have $U_{\sigma}(\gamma) = \sigma_{I_1}(U(\gamma))$ and
\begin{align*}
 z_{\sigma}(\gamma) V
 &= U(\gamma)U_{\sigma}(\gamma)^* V \\
 &= U(\gamma)\sigma_{I_1}(U(\gamma)^*) V \\
 &= U(\gamma) V \rho_{I_1}(U(\gamma)^*) \\
 &= \cAd U(\gamma)(V) \cdot U(\gamma) \rho_{I_1}(U(\gamma)^*) \\
 &= V^\gamma z_{\rho}(\gamma).
\end{align*}
If $\gamma \in \uDiffSone$ is $\gamma = \gamma_1 \gamma_2\cdots \gamma_n$ with $\gamma_j$ localized in some interval $I_j$, then $z_{\sigma}(\gamma_1 \gamma_2\cdots \gamma_n) V = V^{\gamma_1 \gamma_2\cdots \gamma_n} z_{\rho}(\gamma_1 \gamma_2\cdots \gamma_n)$ using the covariance cocycle identity.

This property of the covariance cocycles is a necessary ingredient to show that certain extensions of a conformal net are again diffeomorphism covariant, see e.g.\ \cite{KL04-1, DG18, MTW18, AGT23Pointed}. 

\subsubsection*{Acknowledgements}
We thank Sebastiano Carpi for informing us of \cite{Zellner16}, and
Marcel Bischoff, Yasuyuki Kawahigashi, Roberto Longo and Ryo Nojima for interesting discussions.
Moreover, we thank Milan Niestijl for informing us of the reference \cite{Glo07}.

M.S.A.\! is a Humboldt Research Fellow supported by the Alexander von Humboldt Foundation
and is partially supported by INdAM--GNAMPA Project ``Ricostruzione di segnali, tramite operatori e frame, in presenza di rumore'' CUP E53C23001670001, and by GNAMPA--INdAM.
L.G.\! and Y.T.\! acknowledge support from the GNAMPA--INdAM project \lq\lq Operator algebras and infinite quantum systems" CUP E53C23001670001,
project GNAMPA 2024 \lq\lq Probabilit\`a Quantistica e Applicazioni" CUP E53C23001670001,
from the MIUR Excellence Department Project MatMod@TOV awarded to the Department of Mathematics of the University of Rome Tor Vergata CUP E83C23000330006,
and from the University of Rome Tor Vergata funding OANGQS CUP E83C25000580005.

\def\polhk#1{\setbox0=\hbox{#1}{\ooalign{\hidewidth
  \lower1.5ex\hbox{`}\hidewidth\crcr\unhbox0}}} \def\cprime{$'$}

\end{document}